\documentclass[12pt]{amsart}

\usepackage{pb-diagram}

\usepackage{amsfonts}
\usepackage{amsmath}
\usepackage{amssymb}

\allowdisplaybreaks

\newcommand{\tr}{\textnormal{tr}}
\newcommand{\ric}{\textnormal{Ric}}

\newcommand{\dbar}{\overline{\partial}}

\newcommand{\ddt}[1]{\frac{\partial #1}{\partial t}}

\newcommand{\ddbar}{\sqrt{-1}\partial\dbar}

\newtheorem{theorem}{Theorem}[section]
\newtheorem{proposition}{Proposition}[section]
\newtheorem{lemma}{Lemma}[section]

\newtheorem{corollary}{Corollary}[section]

\renewcommand{\thefootnote}{\fnsymbol{footnote}}
\newcommand{\starttext}{ \setcounter{footnote}{0}
\renewcommand{\thefootnote}{\arabic{footnote}}}

\newcommand{\beq}{\begin{equation}}
\newcommand{\bea}{\begin{eqnarray}}
\newcommand{\eea}{\end{eqnarray}} \newcommand{\ee}{\end{equation}}

\def\ba{\begin{eqnarray}}
\def\ea{\end{eqnarray}}

\def\cS{{\mathcal S}}

\def\ti\tilde

\def\u{\underline}

\def\sub{\subseteq}

\def\tr{{\rm tr}}

\def\ti{\tilde}

\begin{document}

\starttext \baselineskip=18pt \setcounter{footnote}{0}

\title{Diameter estimates in K\"ahler geometry}

\author[{Bin Guo, Duong H. Phong, Jian Song and Jacob Sturm}
]{Bin Guo$^*$, Duong H. Phong$^\dagger$, Jian Song$^\ddagger$ and Jacob Sturm$^{\dagger\dagger}$ }

\thanks{Work supported in part by the National Science Foundation under grants DMS-22-03273 and DMS-22-03607, and the collaboration grant 946730 from Simons Foundation.}

\address{$^*$ Department of Mathematics \& Computer Science, Rutgers University, Newark, NJ 07102}

\email{bguo@rutgers.edu}

\address{$^\dagger$ Department of Mathematics, Columbia University, New York, NY 10027}

\email{phong@math.columbia.edu}

\address{$^\ddagger$ Department of Mathematics, Rutgers University, Piscataway, NJ 08854}

\email{jiansong@math.rutgers.edu}

\address{$^{\dagger\dagger}$ Department of Mathematics \& Computer Science, Rutgers University, Newark, NJ 07102}

\email{sturm@andromeda.rutgers.edu}

\begin{abstract}

{\footnotesize Diameter estimates for K\"ahler metrics are established which require only an entropy bound and no lower bound on the Ricci curvature. The proof builds on recent PDE techniques for $L^\infty$ estimates for the Monge-Amp\`ere equation, with a key improvement allowing degeneracies of the volume form of codimension strictly greater than one. As a consequence, diameter bounds are obtained for long-time solutions of the K\"ahler-Ricci flow and finite-time solutions when the limiting class is big, as well as for special fibrations of Calabi-Yau manifolds.}

\end{abstract}

\maketitle

\section{Introduction} \label{secintro}

\setcounter{equation}{0}

The diameter is one of the most important geometric invariants defined by a metric. Bounds for the diameter are for example essential in the study of convergence of manifolds, which is of particular interest in moduli problems and geometric flows, where one hopes to arrive at a canonical model by taking limits. Unfortunately, in Riemannian geometry, there are very few tools for estimating the diameter, besides comparison theorems and the Bonnet-Myers theorem which require that the Ricci curvature be strictly positive. The situation did not seem markedly different in K\"ahler geometry, although we can exploit there the fact that the potential can be viewed as the solution of a complex Monge-Amp\`ere equation with right hand side given by its volume form (see e.g. \cite{FGS, SW1, L, GPTW1}). However, there has been considerable progress recently in PDE methods for $L^\infty$ estimates for fully non-linear equations \cite{GPT,GP,GP22}. 
These new methods turn out to be particularly amenable to geometric estimates, and have been shown to imply some promising estimates for non-collapse \cite{GS} and for the Green's function \cite{GPS}.

\smallskip
The main goal of the present paper is to develop a general theory of diameter estimates in K\"ahler geometry. We shall be particularly interested in estimates which require only an upper bound for the entropy of the volume form, but not a lower bound for the Ricci curvature. For geometric applications, it is also important that the diameter estimates be uniform with respect to suitable subsets which can approach the boundary of the K\"ahler cone. To obtain such estimates, we build on the PDE methods mentioned above \cite{GPT, GS, GPS}, but with an essential improvement. These methods originally applied to fully non-linear elliptic equations which satisfy a specific structural condition, corresponding to a condition of nowhere vanishing of the volume form in the case of Monge-Amp\`ere. It has been recently shown by Harvey and Lawson \cite{HL} that this condition is natural and applies to very broad classes of equations. Nevertheless, for many applications which we shall consider in this paper, notably the K\"ahler-Ricci flow and the analytic Minimal Model Program, it is necessary to allow the volume form to be arbitrarily close to vanishing along subsets which are suitably small, such as a proper complex subvariety. It turns out that such a generalization is indeed possible, and it plays a major role in this whole paper.

\medskip

We now state our general diameter estimates more precisely. Let $(X, \omega_X)$ be an $n$-dimensional compact K\"ahler manifold equipped with a K\"ahler metric $\omega_X$. Let ${\mathcal{K}}(X)$ be the space of K\"ahler metrics on $X$. We define the $p$-Nash entropy of a K\"ahler metric $\omega\in \mathcal{K}(X)$ associated to $(X, \omega_X)$ by
\begin{equation}
\mathcal{N}_{X, \omega_X, p}(\omega) = \frac{1}{V_\omega}  \int_X \left| \log \left((V_\omega)^{-1}\frac{\omega^n}{\omega_X^n} \right)\right|^p \omega^n, ~ V_\omega= \int_X \omega^n = [\omega]^n, 
\end{equation}
for $p>0$. If we write $e^F= \frac{1}{V_\omega}  \frac{\omega^n}{\omega_X^n}$, then 
$$\mathcal{N}_{X, \omega_X, p} (\omega)= \int_X e^F \left|  F \right|^p \omega_X^n = \left\| e^F \right\|_{L^1\log L^p(X, \omega_X)}$$
and $$\mathcal{N}_{X, \omega_X, p} (\omega) \leq \int_{F\geq 0} e^F F^p \omega_X^n + C$$ for some $C=C(X, \omega_X, p)>0$.

We introduce the following set of admissible functions for given parameters $A, B, K>0$, $p>n$,
\begin{equation} \label{admiss}
{\mathcal V}(X, \omega_X, n, A, p, K ) = \left\{ \omega\in {\mathcal{K}}(X):   [\omega]\cdot[\omega_X]^{n-1}\le A, ~  \mathcal{N}_{X, \omega_X, p}(\omega)  \le K   \right\}. 
\end{equation}
Let $\gamma$ be a non-negative continuous function. 
We further define a subset of ${\mathcal V}(X, \omega_X, n, A, p, K )$ by 
\begin{equation}\label{admiss2}
 {\mathcal  W} (X, \omega_X, n, A, p, K;   \gamma ) = \left\{ \omega \in {\mathcal{V}}(X, \omega_X, n, A, p, K): ~(V_\omega)^{-1}\frac{\omega^n} {\omega_X^n} \geq \gamma \right\}.   
\end{equation}
We also define the Green's function $G(x, y)$ associated to a Riemannian manifold $(X, g)$ by (see e.g. \cite{SY})
$$\Delta_g G(x, \cdot) = -\delta_x(\cdot) + ({\rm Vol}_g(X))^{-1},$$
where $\Delta_g$ is the Laplace operator associated to $g$. The following is the main theorem of our paper:

\begin{theorem}\label{thm:main1} Let $X$ be an $n$-dimensional connected K\"ahler manifold equipped with a K\"ahler metric $\omega_X$ and let $\gamma$ be a nonnegative continuous function on $X$ satisfying 
\begin{equation}\label{mindim}
\dim_{\mathcal{H}}\{\gamma=0\}< 2n-1,~\gamma \geq 0,
\end{equation}
where $\dim_{\mathcal{H}}$ is the Hausdorff dimension. 
Then for any $A, K>0$ and $p>n$, there exist $C=C(X, \omega_X, n, A, p, K, \gamma)>0$, $c=c(X, \omega_X, n, A, p, K, \gamma)>0$ and $\alpha = \alpha(n, p)>0$ such that for any $\omega\in {\mathcal{W}}(X, \omega_X, n, A, p, K; \gamma)$, we have the following bounds for

{\rm (a)} The Green's function:
$$
\int_X |G(x, \cdot)| \omega^n +\int_X |\nabla  G(x, \cdot)| \omega^n + \left( - \inf_{y\in X} G(x, y) \right)  {\rm Vol}_\omega(X) \leq C
$$
 for any $x\in X$;
 
 {\rm (b)} The diameter:
 \begin{equation}
 {\rm diam}(X, \omega) \leq C;\nonumber
 \end{equation}

 {\rm (c)} The volume element:  for any $x\in X$ and any $R\in (0,1]$,
 \begin{equation} 
\frac{{\mathrm{Vol}}_{\omega} (B_{\omega}(x, R))}{\textnormal{Vol}_\omega(X)}\ge c R^\alpha.\nonumber
\end{equation}

\end{theorem}

\bigskip
This is the first general result on uniform diameter bounds and volume non-collapsing estimates for K\"ahler manifolds without any curvature assumption. We note that the assumption on the Hausdorff dimension of the set $\{\gamma=0\}$ can be replaced by the weaker assumption that this set have small measure together with the connectedness of $\{ \gamma>0\}$ (c.f. Proposition \ref{gres!} and Proposition \ref{diaest}), and that this is in fact how the theorem is proved. In practice, the theorem is often applied with $\{\gamma=0\}$ supported on a proper analytic subvariety of $X$. It may be instructive to compare it with the situation in Riemannian geometry. There the sharp result is the theorem of S.Y. Cheng and P. Li \cite{CL}, where a lower bound for the Green's function requires a lower bound on the Ricci curvature. Since $\ric(\omega)=-i\partial\bar\partial \log\,\omega^n$ in K\"ahler geometry, and we allow the lower bound $\gamma$ for the volume form $\omega^n$ to vanish, we see that Theorem \ref{thm:main1} can give lower bounds for the Green's function even when no lower bound for the Ricci curvature is available. As we shall see, this flexibility is particularly important in the study of the K\"ahler-Ricci flow and of fibrations of Calabi-Yau manifolds.

\medskip

As a consequence of Theorem \ref{thm:main1}, we obtain the following theorem, which can be viewed as a K\"ahler analogue of Gromov's precompactness theorem for metric spaces:
\begin{theorem}\label{cor:main1} Let $X$ be an $n$-dimensional connected K\"ahler manifold equipped with a K\"ahler metric $\omega_X$ and let $\gamma$ be a nonnegative continuous functionon $X$ with 
$$
\dim_{\mathcal{H}}\{\gamma=0\}< 2n-1.
$$
Then for any $A, K>0$, $p>n$ and any sequence $\{\omega_j \}_{j=1}^\infty \subset  {\mathcal  W}(X, \omega_X, n, A, p, K;  \gamma ) $,  after passing to a subsequence, $(X, \omega_j)$ converges in Gromov-Hausdorff topology to a compact metric space $(X_\infty, d_\infty)$.

\end{theorem}

Geometric compactness is  fundamental for understanding degeneration and moduli problems for complex Riemannian manifolds. Curvature bounds are usually necessary such as in the general theory of Cheeger-Colding \cite{CC}.  Theorem \ref{cor:main1} bypasses the curvature requirement to provide boundedness for families of K\"ahler manifolds. It might also be combined with techniques from \cite{CCT}  to explore formation of singularities and achieve stronger geometric regularity for the limiting metric spaces.

\section{Applications to the K\"ahler-Ricci flow and cscK} \label{secapp}
\setcounter{equation}{0}

We describe now the applications of Theorem \ref{thm:main1} to the K\"ahler-Ricci flow and  constant scalar curvature K\"ahler metrics.
We shall be particularly interested in the analytic Minimal Model Program introduced in \cite{ST} in relation to the formation of both finite time and long time singularities in the K\"ahler-Ricci flow. 

\medskip

We first consider the following unnormalized K\"ahler-Ricci flow on a K\"ahler manifold $X$ with an initial K\"ahler metric $g_0$ 
 \begin{equation}\label{krflow}
\left\{
\begin{array}{l}
{ \displaystyle \ddt{g} = -\ric(g),}\\
\\
g|_{t=0} =g_0 .
\end{array} \right.
\end{equation}

Let 
\begin{equation}\label{singT}
T= \sup\{ t >0~|~ [g_0]+t[K_X] >0 \} \in \mathbb{R} \cup\{\infty\}.
\end{equation}
 It is shown in \cite{Ts, TZ} that the K\"ahler-Ricci flow has a maximal solution $g(t)$ for $t\in [0, T)$.

\subsection{The case of finite-time singularities}

If $T<\infty$, the flow (\ref{krflow}) must develop singularities at $t=T$. In this case, $\alpha_T= [g_0]+T[K_X]$ is a nef class on $X$.  If $[g_0]\in H^2(X, \mathbb{Q})$, $\alpha_T$ is semi-ample by Kawamata's base point free theorem and the numerical dimension of $\alpha_T$ coincides with its Kodaira dimension. In general, it is unclear if there exists a smooth semi-positive closed $(1,1)$-form in $\alpha_T$. The most interesting case is when $\alpha_T$ is big, i.e., there exists a K\"ahler current in $\alpha_T$ or equivalently, $(\alpha_T)^n >0$. Such a bigness condition can be also interpreted by the total volume along the K\"ahler-Ricci flow as  
\bea
\label{big}
(\alpha_T)^n= \lim_{t\rightarrow T} {\rm Vol}_{g(t)}(X) >0.
\eea
This bigness assumption for the limiting class is in fact generic for finite-time singularities.

It is conjectured in \cite{ST} as part of the analytic minimal model program that when $\alpha_T$ is big, $(X, g(t))$ should converge to a compact K\"ahler variety and the flow will extend uniquely through the singular time through a canonical metric surgery. Such a surgery corresponds to either a divisorial contraction or a flip in birational geometry.  
After suitable blow-ups, the singularity model is expected to be a transition from a shrinking soliton to an expanding soliton (c.f. \cite{S0}). This is confirmed in the case of K\"ahler surfaces by \cite{SW1, SW2}, where it is shown that the K\"ahler-Ricci flow contracts finitely many holomorphic $S^2$ of $(-1)$ self-intersection in Gromov-Hausdorff topology. 
The following diameter bound is the first step to understand formation of finite time singularities of the K\"ahler-Ricci flow in general dimension:

\begin{theorem}\label{thm:main2} Let $(X, g_0)$ be a K\"ahler manifold equipped with a K\"ahler metric $g_0$. If $g(t)$ is the maximal solution of the K\"ahler-Ricci flow (\ref{krflow}) for $t\in [0, T)$ for some $T\in \mathbb{R}^+$ and if the limiting class $[g_0]+T [K_X]$ is big, then there exist $C=C(X, g_0)>0$, $c=c(X, g_0)>0$ and $\alpha=\alpha(X, g_0)>0$ such that for any $t\in [0, T)$, 
$${\mathrm{diam}}(X, g(t)) \leq C,$$
$$
\int_X |G_t(x, \cdot)| dV_{g(t)} +\int_X |\nabla  G_t(x, \cdot)| dV_{g(t)} +  \left( - \inf_{y\in X} G_t(x, y) \right)  {\rm Vol}_{g(t)} (X) \leq C,$$
$$\frac{{\mathrm{Vol}}_{g(t)} (B_{g(t)}(x, R))}{\textnormal{Vol}_{g(t)}(X)}\ge c R^\alpha, $$
for any $x\in X$ and $R\in (0,1]$, where $G_t$ is the Green's function for $(X, g(t))$.
 
\end{theorem}

We stress that Theorem \ref{thm:main2} holds for general K\"ahler manifolds, and no projectiveness assumption is needed. We prove it in section \ref{seckrff} by establishing a uniform upper bound for the $p$-Nash entropy for any $p>0$ and a lower bound for the volume form along the flow.

%

\subsection{The case of long-time solutions}

It is well-known that the K\"ahler-Ricci flow has a long-time solution if and only if the canonical bundle $K_X$ is nef. The underflying manifold $X$ is then called a minimal model. The Kodaira dimension of $X$ is defined by
$${\rm Kod}(X) = \lim_{m\rightarrow \infty} \frac{\log h^0(X, mK_X)}{\log m}$$
if $h^0(X, mK_X)\neq 0$ for some $m\in \mathbb{Z}^+$. 
The Kodaira dimension of $X$ is always no greater than $n$ and is nonnegative as long as there exists one holomorphic pluricanonical section. The abundance conjecture predicts that if $X$ is minimal, $K_X$ must be semi-ample and hence the Kodaira dimension is always nonnegative. We would like now to obtain a uniform diameter bound for long time solutions of the K\"ahler-Ricci flow.

\smallskip

We consider the following normalized K\"ahler-Ricci flow with initial metric $g_0$ if the Kodaira dimension of $X$ is nonnegative. 
\begin{equation}\label{nkrflow}
\left\{
\begin{array}{l}
{ \displaystyle \ddt{g} = -{\rm Ric} (g) - g,}\\
\\
g|_{t=0} =g_0 .
\end{array} \right.
\end{equation}
Obviously, the flow (\ref{nkrflow}) exists for $t\in [0,\infty)$ since $K_X$ is nef.

\begin{theorem}\label{thm:main3} Let $(X, g_0)$ be an $n$-dimensional K\"ahler manifold with nef $K_X$ and non-negative Kodaira dimension. Let $g(t)$ be the solution of the normalized K\"ahler-Ricci flow (\ref{nkrflow}). 
Then there exist $C=C(X, g_0)>0$, $c=c(X, g_0)>0$ and $\alpha=\alpha(X, g_0)>0$ such that for any $t\geq 0$, 
$${\mathrm{diam}}(X, g(t)) \leq C,$$
$$
\int_X |G_t(x, \cdot)| dV_{g(t)} +\int_X |\nabla  G_t(x, \cdot)| dV_{g(t)} +  \left( - \inf_{y\in X} G_t(x, y) \right)  {\rm Vol}_{g(t)} (X) \leq C,$$
$$\frac{{\mathrm{Vol}}_{g(t)} (B_{g(t)}(x, R))}{\textnormal{Vol}_{g(t)}(X)}\ge c R^\alpha, $$
for any $x\in X$ and $R\in (0,1]$, where $G_t$ is the Green's function for $(X, g(t))$.

\end{theorem}

\bigskip

The diameter is optimal for $X$ with positive Kodaira dimension. When $c_1(X)=0$ and hence $\kappa(X)=0$, the diameter of $(X, g(t))$ decays at the exact rate $e^{-t/2}$. The uniform diameter bound in Theorem \ref{thm:main3} is proved in  \cite{JS} in the case when $K_X$ is semi-ample, and the proof is built on works of \cite{STZ, Ba}, relying on the uniform scalar curvature bound obtained in \cite{ST1, Zh} (see also \cite{W} in the case of general type). Our proof in the present case is based rather on Theorem \ref{thm:main1}. This has many advantages, since we do not need then any assumption on the scalar curvature (for which bounds are not available in the nef case), nor on the projectiveness of $X$, nor on the abundance conjecture. 

\medskip
Next, we discuss the behavior of the flow near a singular fiber in the case of collapse.
If $K_X$ is semi-ample, the pluricanonical system of $X$ induces a unique holomorphic fibration 
$$\pi: X \rightarrow X_{can}, $$
where $X_{can}$ is the unique canonical model of $X$
\footnote{More generally, the abundance conjecture predicts that $K_X$ is nef if and only if it is semi-ample when $X$ is projective.}. The Kodaira dimension of $X$ coincides with the complex dimension of $X_{can}$. The general fibre of $\pi$ is a smooth Calabi-Yau manifold. It is proved in \cite{ST0, ST01} that the normalized K\"ahler-Ricci flow converges weakly to a twisted possibly singular K\"ahler-Einstein metric $g_\infty$ on $X_{can}$ satisfying 
$${\rm Ric}(g_\infty) = - g_\infty+g_{WP},$$
where $g_{WP}$ is the Weil-Petersson metric for the Calabi-Yau fibration $\pi: X \rightarrow X_{can}$, while the fibre metrics collapses along the normalized K\"ahler-Ricci flow. In particular, if $X_y = \pi^{-1}(Y)$ is a smooth fibre, it is shown in \cite{TWY} that  $e^t g(t)|_{X_y}$ converges to the unique Ricci-flat K\"ahler metric in $[g(0)|_{X_y}]$. The following theorem describe the asymptotic behavior near the singular fibre with at worst canonical singularities.  

\begin{theorem} \label{krfl2} Let $(X, g_0)$ be an $n$-dimensional  projective manifold with semi-ample $K_X$ and ${\rm Kod}(X)=1$. Let $g(t)$ be the solution of  the normalized K\"ahler-Ricci flow (\ref{nkrflow}).  If every fibre of $\pi: X \rightarrow X_{can}$  has at worst canonical singularities, then  there exists $C=C(X, g_0)$ such that for all $t\geq 0$ and $y\in X_{can}^\circ$, we have 
\begin{equation}\label{fbestin}
 {\rm diam}(X_y, g(t)|_{X_y}) \leq Ce^{-\frac{t}{2}}.
\end{equation}

\end{theorem}

\bigskip

The fibre diameter estimate (\ref{fbestin}) is intrinsic as the diameter is achieved by a minimal geodesic in the fibre. It immediately implies that the extrinsic fibre diameter estimate holds uniformly for all fibres of $\pi: X \rightarrow X_{can}$ since $X_{can}^\circ$ is an open dense subset of $X_{can}$. Theorem \ref{krfl2} can be compared with the diameter estimates in Li \cite{L, L2} for fibres of collapsing Ricci-flat K\"ahler metrics on a projective Calabi-Yau manifold. In fact, the proof of Theorem \ref{krfl2} based on Theorem \ref{thm:main1} can   serve as an alternative proof and improvement of the diameter estimates in \cite{L, L2} (c.f. Theorem \ref{10cyf}). One can further derive uniform bounds for the Green's function and volume non-collapsing on each smooth fibre with respect to the rescaled metric $e^t g(t)$ (c.f. Theorem \ref{10krfl}). 

\subsection{The case of constant scalar curvature K\"ahler metrics}

Finally, we would like to extend our results to families of cscK metrics with bounded $p$-Nash entropy for $p\leq n$. We will apply Theorem \ref{thm:main1} to cscK metrics on smooth minimal models. 

\begin{theorem}  \label{thm:maincsc} 
  Let $X$ be an $n$-dimensional smooth minimal model of general type. For any K\"ahler class $\mathcal{A}$ of $X$, there exist $\delta_0=\delta_0(\mathcal{A})>0$, $C=C(\mathcal{A}, \delta_0)>0$, $\alpha=\alpha(\mathcal{A}, \delta_0)$ and $c=c(\mathcal{A}, \delta_0)>0$ such that  for any $0<\delta<\delta_0$,  there exists a unique cscK metric $\omega_\delta$ in $K_X+\delta \mathcal{A}$ satisfying
$${\rm diam}(X, \omega_\delta) \leq C,$$
$$ \int_X |G_\delta(x, \cdot)| \omega_\delta^n+\int_X |\nabla  G_\delta(x, \cdot)| \omega_\delta^n  +   \left(- \inf_{y\in X} G_\delta(x, y) \right)  {\rm Vol}_{\omega_\delta} (X) \leq  C , 
$$
$$
\frac{{\mathrm{Vol}}_{\omega} (B_{\omega_\delta}(x, R))}{\textnormal{Vol}_{\omega_\delta} (X)}\ge c R^\alpha, $$
for any $x\in X$ and $R\in (0,1)$, where $G_\delta(x,y)$ is the Green's function of $(X, \omega_\delta)$.

In particular, if the canonical model $X_{can}$ of $X$ has only isolated singularities, then $(X, \omega_\delta)$ converges to the K\"ahler-Einstein metric space $(X_{can}, d_{KE})$ in Gromov-Hausdorff topologgy as $\delta \rightarrow 0$, where $(X_{can}, d_{KE})$ is the metric completion of the unique smooth K\"ahler-Einstein metric on the regular part of $X_{can}$ in \cite{S1}. 

\end{theorem}

The existence of cscK metrics in a K\"ahler class near $[K_X]$ on a minimal model is proved in \cite{We, SW0, JSS, SD, S2}. When $X$ is a minimal model of general type, there exists a unique K\"ahler-Einstein current with bounded potentials by the work of \cite{K, EGZ, Zh0} and it is a smooth K\"ahler-Einstein metric $g_{KE}$ on $X_{can}^\circ$, the regular part of $X_{can}$. It is proved in \cite{S1} that the metric completion of $(X_{can}^\circ, g_{KE})$ is a compact metric space homeomorphic to $X_{can}$ itself.  

It is conjectured in \cite{JSS} that if $X$ is a minimal model, then the cscK metric spaces near the canonical class $[K_X]$ converge geometrically to the twisted K\"ahler-Einstein space $(X_{can}, d_{can})$. Theorem \ref{thm:maincsc} can be viewed as a partial confirmation of this conjecture. 


\section{Bounded sets in the K\"ahler cone} \label{seccone}

\setcounter{equation}{0}

\noindent {\em Notational convention}: if $\omega = (g_{i\bar j})$ is a K\"ahler metric and $\theta = (\theta_{i\bar j})$ is a $(1,1)$-form, we denote $\tr_\omega (\theta) = g^{i\bar j}\theta_{i\bar j}$, where $
(g^{i\bar j})$ is the inverse of $(g_{i\bar j})$. For a number $p\in (1,\infty)$, we denote by $p^*$ the conjugate exponent of $p$, i.e., $\frac 1p + \frac{1}{p^*} = 1$.

\begin{proposition} \label{cohobd}
Let $(X,\omega_X)$ be an $n$-dimensional compact K\"ahler manifold equipped with a K\"ahler metric $\omega_X$ in a K\"ahler class $\alpha$. For any $k\geq 0$, and a cohomology class $\beta \in H^{1,1}(X,\mathbb R)$, there exists a smooth representative $\theta \in \beta$ such that 
\begin{equation}\label{kbd}
\|\theta\|_{C^{k}(X, \omega_X)}\leq C, 
\end{equation}
for some constant $C=C(X, \omega_X,  k, |\beta\cdot \alpha^{n-1}|, |\beta^2\cdot 
\alpha^{n-2}| )>0$. 
\end{proposition}

\begin{proof}  We will take $\theta$ to be the unique harmonic $(1,1)$-form in the class $\beta$, relative to the K\"ahler metric $\omega_X$, i.e. 
$$\Delta_{\bar \partial} \theta = 0.$$ 
By the standard Bochner-Kodaira-Lichnerowicz formula, we have 
\begin{equation}\label{eqn:BKL1}
0=-\Delta_{\bar \partial}\theta_{i\bar k} = \frac{1}{2}( \theta_{i\bar k, j\bar j} + \theta_{i\bar k, \bar j j}  ) + \theta_{m\bar j} R_{i\bar m j \bar k} - \frac 12 \theta_{m\bar k} R_{i\bar m} - \frac 12 \theta_{i\bar m} R_{m\bar k},
\end{equation}
where $R_{i\bar m j \bar k}$, $R_{i\bar m}$ denote the Riemann and Ricci curvatures of the fixed metric $\omega_X$, and $\theta_{i\bar k, j\bar j}$, $\theta_{i\bar k, \bar j  j}$ denote the covariant derivatives of $\theta_{i\bar k}$ with respect to the connection induced by $\omega_X$. It is well-known that equation (\ref{eqn:BKL1}) is a linear elliptic equation of the $(1,1)$-form $\theta_{i\bar k}$.

\medskip

Taking traces on both sides of (\ref{eqn:BKL1}), we obtain 
$$\Delta_{\omega_X} \left( \tr_{\omega_X} \theta \right)= 0, $$
 hence $\tr_{\omega_X}\theta$ must be a constant. Then we have
$$c_1 = \int_X \theta\wedge \omega_X^{n-1} = \frac{V}{n}\tr_{\omega_X}\theta,$$
where $V = \int_X \omega_X^n =[\omega_X]^n$.   On the other hand, we also have
\begin{equation}\label{c2}
c_2 = \int_X \theta^2 \wedge( \omega_X)^{n-2} = C(n)\int_X ( (\tr_{\omega_X}\theta)^2 - |\theta|_{\omega_X}^2  ) \omega_X^n.
\end{equation}
From this we see that 
\begin{equation}\label{theta3}
\int_X |\theta|_{\omega_X}^2  \omega_X^n
\end{equation}
is uniformly bounded depending only on $c_1$, $c_2$ and $\omega_X$. Applying Moser iteration to the equation (\ref{eqn:BKL1}) we get an $L^\infty$ bound for $\theta$. The uniform $C^{k, \alpha}$ estimates of $\theta$ then follow from the standard elliptic estimates applied to the linear elliptic equation (\ref{eqn:BKL1}). 
\end{proof}

Note that when $\beta$ is K\"ahler, $c_2$ in (\ref{c2}) is positive and so (\ref{theta3}) is uniformly bounded depending only on $c_1$ and  $\omega_X$. The following corollary is then an immediate consequence of Proposition \ref{kbd}.

\begin{corollary} \label{abd} Let $(X, \omega_X)$ be an $n$-dimensional K\"ahler manifold equipped with a K\"ahler metric $\omega_X$. Then the following hold.

\begin{enumerate}

\item For any bounded set $U$ in the K\"ahler cone of $X$, there exists $C=C(U)>0$ such that for any K\"ahler class $\beta \in U$, 
$$\beta\cdot [\omega_X]^{n-1}< C. $$

\item For any $A>0$ and a K\"ahler class $\beta$ with 
$$\beta\cdot [\omega_X]^{n-1} < A, $$ 
there exists $C=C(A)>0$ such that 
$$C [\omega_X] - \beta$$
is a K\"ahler class.

\end{enumerate}

\end{corollary}

\begin{corollary} \label{albd} Let $(X, \omega_X)$ be an $n$-dimensional K\"ahler manifold equipped with a K\"ahler metric $\omega_X$. For any bounded set $U$ in the K\"ahler cone of $X$, there exists a smooth closed $(1,1)$-form $\theta\in \beta$ for any $\beta \in U$ with the following uniform properties.
\begin{enumerate}

\item There exists $C=C(U)>0$ such that $ \|\theta \|_{C^3(X, \omega_X)} \leq C, $

\medskip

\item There exist $\alpha=\alpha(U)>0$ and $C=C(U, \alpha)>0$ such that for any $\varphi \in {\rm PSH}(X, \theta)$, 
$$\int_X e^{- \alpha(\varphi - \sup_X\varphi)} (\omega_X)^n \leq C, $$

\end{enumerate}

\end{corollary}

\begin{proof} It suffices to prove (2). By Proposition \ref{cohobd}, there exists $B=B(U)>0$ such that $ \theta < B \omega_X$. Then for any $\varphi\in {\rm PSH}(X, \theta)$, we have $\varphi \in {\rm PSH}(X, B\omega_X)$. The corollary then follows by applying the $\alpha$-invariant of $B\omega_X$.
\end{proof}


\section{$L^\infty$-estimates for complex Monge-Amp\`ere equations}\label{seclinfty}

\setcounter{equation}{0}

\begin{proposition} \label{linf} Let $(X, \omega_X)$ be an $n$-dimensional compact K\"ahler manifold equipped with a K\"ahler metric $\omega_X$. For any $K>0$, $p>n$, there exist $C=C(X,\omega_X, n, p, K)>0$ such that if $\theta$ a smooth closed $(1,1)$-form with $\theta \leq \omega_X$ and if $\omega =\theta + \ddbar \varphi$ is a K\"ahler form satisfying
$$\mathcal{N}_{X, \omega_X, p}(\omega) \leq K, $$
then   
$$
 \|\varphi - \sup_X \varphi - \mathcal{V}_\theta  \|_{L^\infty(X)} \leq C, 
$$
where $\mathcal{V}_\theta = \sup\{ u: u\in {\rm PSH}(X, \theta), ~ u\leq 0\} $  is the envelope of non-positive $\theta$-psh functions.

\end{proposition}

\begin{proof} By the assumption on $\theta$, we have ${\rm PSH}(X, \theta)\sub {\rm PSH}(X, \omega_X)$. Therefore there exists $\alpha>0$ and $C_\alpha>0$ such that for any $\varphi \in {\rm PSH}(X, \theta)$, 
$$\int_X e^{-\alpha(\varphi - \sup_X \varphi)} \omega_X^n \leq C_\alpha.$$ Then the proposition follows from the uniform $L^\infty$-estimates from \cite{EGZ, DP, BEGZ, FGS, GPTW} which generalize Kolodziej's $L^\infty$ estimates in the case of a fixed background metric \cite{K}. 
\end{proof}

\begin{corollary} \label{linf2} Let $(X, \omega_X)$ be an $n$-dimensional compact K\"ahler manifold equipped with a K\"ahler metric $\omega_X$. For any $A, K>0$ and $p>n$, if two K\"ahler forms $\omega_1$ and $\omega_2$ belong to the same K\"ahler class and 
$$ \omega_1, \omega_2 \in \mathcal{V}(X,\omega_X, n, A, p, K),$$ 
then there exist $C=C(X,\omega_X, n, A, p, K)>0$ and $\varphi \in C^\infty(X)$ such that 
\begin{equation}
\omega_2=\omega_1+\ddbar \varphi, ~  \|\varphi - \sup_X \varphi \|_{L^\infty(X)} \leq C. 
\end{equation}

\end{corollary}

\begin{proof} Let $\gamma$ be the K\"ahler class of $\omega_1$ and $\omega_2$. Then $\gamma\cdot [\omega_X]^{n-1} <A$ by our assumption. By Proposition \ref{cohobd}, we can choose a smooth closed $(1,1)$-form $\theta\in \gamma$ (not necessarily positive) such that $$ \|\theta \|_{C^3(X, \omega_X)}$$ is uniformly bounded by a constant that only depends on $A$. Furthermore, We define $\varphi_i\in C^\infty(X)$ by  
$$\omega_i = \theta + \ddbar \varphi_i, ~ \sup_X \varphi_i =0, ~i=1, 2. $$
Then $\varphi_i$ satisfies the following complex Monge-Amp\`ere equation
\begin{equation}
\frac{(\theta + \ddbar \varphi_i)^n}{ [\theta]^n } = e^{F_i} (\omega_X)^n, ~ \sup \varphi_i=0.
\end{equation}
Since $\omega_i \in \mathcal{V}(X, \omega_X, n, A, p, K)$, 
$$ \|e^{F_i} \|_{L^1\log L^p(X, \omega_X)} \leq K, ~i=1, 2.$$

Let $\mathcal{V}_\theta= \sup\{ u: u\in {\rm PSH}(X, \theta), ~u\leq 0\}$.  Proposition \ref{linf} implies that there exists $C=C(X, \omega_X, n, A, p, K) >0$ such that
$$
 \| \varphi_i - \mathcal{V}_\theta  \|_{L^\infty(X)} \leq C, 
$$
and so 
$$ \| \varphi_1 - \varphi_2  \|_{L^\infty(X)} \leq 2C.$$ The corollary is proved, with $\varphi=\varphi_1-\varphi_2$. 
\end{proof}


\section{Estimates for the Green's functions} \label{secgreen}

\setcounter{equation}{0}

Throughout this section, we fix the $n$-dimensional K\"ahler manifold $(X, \omega_X)$ and constants $A, K>0$ and $p>n$.   

The following lemma is a natural extension of Lemma 2 in \cite{GPS}.

\begin{lemma} \label{lemma key}
Suppose $\omega\in \mathcal{V}(X, \omega_X, n, A, p, K)$. Let $v\in L^1(X,\omega^n)$ be a function that satisfies $\int_X v \omega^n= 0$ and 
\begin{equation}\label{eqn:assumption 1}
v\in C^2(\overline{\Omega_{0}}),\quad \Delta_{\omega} v \ge -a \mbox{ in }\Omega_{0}
\end{equation}
for some $a> 0$ and $\Omega_s = \{v> s\}$ is the super-level set of $v$. Then there is a uniform constant $C=C(X, \omega_X, A, p, K)>0$  such that 
\begin{equation}
\label{inequality}
\sup_X v\le C(a+ \frac{1}{[\omega]^n} \int_X |v| \omega^n).
\end{equation}
\end{lemma}

For the convenience of the readers, we sketch the proof of Lemma \ref{lemma key}.

\begin{proof} We follow closely the arguments in \cite{GPS}.

\smallskip
First, we observe that it suffices to prove the lemma in the case $a=1$. This is because both the equation (\ref{eqn:assumption 1}) and the desired inequality (\ref{inequality}) are homogenous under a simultaneous rescaling of $a$ and $v$, $a\to 1$, $v\to {v\over a}$.

\smallskip
Next, we may assume $\| v\|_{L^1(X,\omega^n)}\le [\omega]^n$, otherwise, replace $v$ by $\hat v: =  v \cdot [\omega]^n /\| v\|_{L^1(X,\omega^n)}$ which still satisfies \eqref{eqn:assumption 1} with the same $a=1$ and $\| \hat v\|_{L^1(X,\omega^n)}= [\omega]^n$. It thus suffices to show $\sup_X v\le C$ for some $C>0$ with the dependence as stated in the lemma.   By Proposition \ref{cohobd}, we can choose a smooth closed $(1,1)$-form $\theta\in [\omega]$ with $ \|\theta \|_{C^3(X, \omega_X)}$  uniformly bounded. We let $\omega = \theta + \ddbar \varphi$ for $\varphi\in {\rm PSH}(X,\theta)$ and $\sup_X\varphi = 0$.

\bigskip

\noindent{\bf (1)}. We fix a sequence of positive smooth functions $\eta_k: {\mathbb R}\to {\mathbb R}_+$ such that $\eta_k(x)$ converges uniformly and monotonically decreasingly to the function $x\cdot \chi_{\mathbb R_+}(x)$, as $k\to \infty$. We may choose $\eta_k(x) \equiv 1/k$ for any $x\le -1/2$.  As in \cite{GPT} we make use of auxiliary Monge-Amp\`ere equations.
More precisely, for each $s\ge 0$ and large $k$, we  consider the following specific complex Monge-Amp\`ere equations
\cite{GPS}
\begin{equation}\label{eqn:aMA}
(\theta+ \ddbar \psi_{ s,k}) ^n = [\omega]^n \frac{\eta_k(v - s)}{A_{s, k}} e^F \omega_X^n,\quad \sup_X \psi_{s,k } = 0,
\end{equation}
where 
\begin{equation}\label{eqn:key 1}
A_{s, k} = \int_X {\eta_k(v - s)} e^F \omega_X^n\to  \int_{\Omega_s} (v- s) e^F \omega_X^n=: A_s \mbox{ as }k\to \infty.
\end{equation}
We have assumed that the open set $\Omega_s\neq \emptyset$ so $A_s>0$. The assumption that $\| v\|_{L^1(X,\omega^n)}/[\omega]^n\le 1$ implies that $A_s\le 1$, hence $A_{s,k}\le 2$ for large $k$.

\bigskip

\noindent{\bf (2)}. Recall that we have assumed that $a=1$. We denote $\Lambda = C + 1$ where $C>0$ is the constant in Corollary \ref{linf2}. Consider the test function 
$$\Phi: = - \varepsilon (-\psi_{s,k} + \varphi + \Lambda)^{\frac n{n+1}} + (v-s),$$ where $\varepsilon>0$ is chosen such that
\begin{equation}\label{eqn:epsilon}
\varepsilon^{n+1} = \Big(\frac{n+1}{n^2} \Big)^n (1+\varepsilon n)^n A_{s,k}.
\end{equation}
It follows easily from $A_{s,k}\le 2$ and equation \eqref{eqn:epsilon} that \begin{equation}\label{eqn:new epsilon}\varepsilon\le C(n) A_{s,k}^{1/(n+1)},\end{equation} for some $C(n)>0$ depending only on $n$.
The function $\Phi$ is a  $C^2$ function on $\Omega_{0}$ and 
$
-\psi_{t,k} + \varphi_t + \Lambda\ge 1.$ As shown in \cite{GPS}, it follows from the maximum principle, the choice of $\varepsilon$ in \eqref{eqn:epsilon} and the equations of $\psi_{s,k}$ and $\varphi$ 
 that $\Phi\le 0$ on $X$. 

\bigskip

\noindent{\bf (3)}. From $\Phi\le 0$ and \eqref{eqn:new epsilon}, we have $(v-s) A_{s,k}^{-1/(n+1)} \le C_1 (-\psi_{s,k } + \varphi + \Lambda)^{n/(n+1)}$ on $X$, for some $C_1>0$ depending only on $n$. This together with the $\alpha$-invariant and H\"older-Young inequality (see \cite{GPS} for more details) implies that  for some uniform constant $C_2>0$
\begin{equation}\label{eqn:key 5} r\phi(s+r) \le C_2 \phi(s)^{1+\delta_0},\quad \text{for } \forall s\ge 0 \, \text{and } \forall r>0,\end{equation}
where we denote $\delta_0 = \frac{p-n}{np}>0$ and  $\phi(s) = \int_{\Omega_s}e^F \omega_X^n$. 

The assumption that $\| v\|_{L^1(X,\omega^n)}\le [\omega]^n$ implies that
$\int_{\Omega_0} v e^F \omega_X^n\le 1$. 
Hence for any $s>0$ we have
\begin{equation}\label{eqn:key 6}
\phi(s) = \int_{\Omega_s} e^F \omega_X^n \le \frac{1}{s}\int_{\Omega_0}v e^F \omega_X^n \le \frac{1}{s}.
\end{equation}
We can pick  $s_0 =  (2C_2)^{1/\delta_0} $ to ensure that $\phi(s_0)^{\delta_0} \le 1/(2C_2)$. Given \eqref{eqn:key 5}, we can apply the De Giorgi type iteration argument of Kolodziej \cite{K} to conclude that $\phi(s) = 0$ for any $s> S_\infty$ with 
$$S_\infty = s_0 + \frac{1}{1-2^{-\delta_0}}=(2C_2)^{1/\delta_0} +\frac{1}{1-2^{-\delta_0}} .$$
This means that $\sup_X v\le S_\infty$ and the lemma is proved. 
\end{proof}

\smallskip


The following corollary is an immediate consequence of Lemma \ref{lemma key}.

\begin{corollary} \label{lemma 3}  
Suppose $\omega\in \mathcal{V}(X, \omega_X, n, A, p, K)$. If $v\in C^2(X)$ satisfies
$$|\Delta_{\omega} v| \le 1\, \mbox{ and } \int_X v \omega^n = 0,$$
then there is a uniform constant $C=C(X, \omega_X, n, A, p, K)>0$ such that
$$\sup_X |v|\le C(1 + \frac{1}{[\omega]^n} \int_X |v| \omega^n ).$$

\end{corollary}

In order to bound  $ \frac{1}{[\omega]^n} \int_X |v| \omega^n $, we will have to impose a uniform lower bound for the normalized volume form $$([\omega]^n)^{-1}\frac{ \omega^n}{ \omega_X^n}.$$
In particular, we will consider $\omega\in \mathcal{W}(X, \omega_X, n, A, p, K, \gamma)$ for some nonnegative continuous function $\gamma$.

 \begin{lemma}\label{conn} Suppose $\gamma\geq 0$ is a continuous function on $X$ with $\{ \gamma >0 \}$ being connected. Then for any open subset $V \subset \subset \{\gamma>0\}$, there exists a connected open subset $U$ of $X$ with $$V\subset \subset U \subset\subset \{\gamma>0\}. $$
 
 
 \end{lemma}

\begin{proof} Obviously, $\{\gamma>0\}$ is path connected since it is open and connected. We choose a fixed base point $p\in V$. For any $q\in \overline{V}$, there exists a continuous path $\mathcal{C}$ joining $p$ and $q$ in $\{\gamma >0\}$. We can find an open tubular neighborhood $U_q$ of $\mathcal{C}$ such that $U_q \subset\subset \{\gamma>0\}$. Then 
$$\overline{V} \subset \cup_{q\in \overline{V}} U_q$$ 
and we can find finitely many $q_1, q_2, ..., q_N \subset \overline{V}$ such that 
$$\overline{V} \subset U=\cup_{j=1}^N U_{q_j}.$$ 
Then $U\subset \subset \{\gamma>0\} $ is open and connected since every  $U_{q_j}$ is path connected with a common point $p$. The lemma is then proved.  
\end{proof}

\begin{lemma} \label{lemma 6} There exists $\varepsilon=\varepsilon(X, \omega_X, n, A, p, K)>0$ such that if 
\begin{enumerate}

\item $\gamma\geq 0$ is a continuous function on $X$ such that 
$$|\{\gamma=0\}|_{\omega_X} \leq \epsilon, ~ \{\gamma >0 \}~ \text{ is  connected,}$$

\item $\omega\in \mathcal{W}(X, \omega_X, n, A, p, K; \gamma)$,

\item  $v\in C^2(X)$ satisfies
$$|\Delta_{\omega} v| \le 1\, \mbox{ and } \int_X v \omega^n = 0,$$
\end{enumerate} 
then there exists $C=C(X, \omega_X, n, A, p, K, \varepsilon, \gamma)>0$ such that
$$\frac{1}{[\omega]^n} \int_X |v| \omega^n\le C.$$

\end{lemma}

\begin{proof} The proof is by contradiction. Suppose Lemma \ref{lemma 6} fails.  Then there exist a sequence of $\omega_j \in \mathcal{W}(X, \omega_X, n, A, p, K; \gamma)$ and $v_j\in C^2(X)$ satisfying 
$$\Delta_{\omega_{j}} v_j = h_j,\quad \int_X v_j \omega_{j}^n = 0,$$
for some $|h_j|\le 1$, and as $j\to\infty$
\bea
\frac{1}{[\omega_j]^n} \int_X |v_j| \omega_{j}^n = : N_j \to \infty.
\eea
We define $F_j$ by 
$$e^{ F_j} = \left([\omega_j]^n \right)^{-1} \frac{\omega_j^n}{\omega_X^n}$$
and immediately we have
$$e^{F_j} \geq \gamma. $$

\noindent We now consider $\tilde v_j$ defined by $\tilde v_j = v_j/N_j$.  Clearly we have
$$|\Delta_{\omega_{j}} \tilde v_j| = |h_j|/ N_j\to 0,~  \frac{1}{[\omega_j]^n} \int_X |\tilde v_j| \omega_{j}^n = 1.$$

\noindent Applying Lemma \ref{lemma 3} to $\tilde v_j$, there exists a uniform $C>0$ such that for all $j\geq 1$,  
\begin{equation}\label{vjbd}
\sup_X |\tilde v_j|\le C.
\end{equation}
From the equation of $\tilde v_j$ and integration by parts, we see that
\begin{equation}\label{smallc}
 \int_X |\nabla \tilde v_j|_{\omega_{j}}^2 e^{F_{j}} \omega_X^n=\frac{1}{[\omega_j]^n}\int_X |\nabla \tilde v_j|_{\omega_{j}}^2 \omega_{j}^n  \to 0.  
 \end{equation}

Let $U_0=\{\gamma >0\}$. Suppose $|\{ \gamma=0\}|_{\omega_X} < \varepsilon$ for some sufficiently small $\varepsilon>0$ to be determined. Then for any $\varepsilon>0$, we can pick open connected subsets $U_{3\varepsilon} \subset \subset U_{2\varepsilon} \subset \subset\{\gamma>0\}$ as in Lemma \ref{conn} such that 
 $$|X\setminus U_{2\varepsilon}|_{\omega_X} < 2\varepsilon, ~ |X\setminus U_{3\varepsilon}|_{\omega_X} < 3\varepsilon.$$ 
Without loss of generality, we can assume both $U_{2\varepsilon}$ and $U_{3\varepsilon}$ have smooth boundaries. Let 
\begin{equation}
\delta_\varepsilon= \inf_{U_{2\varepsilon}} \gamma. 
\end{equation}
Then we have  $\inf_{U_{2\varepsilon}} e^{F_j} \geq \delta^{-1}$  and 
\begin{eqnarray*}
\int_{ U_{2\varepsilon}}  |\nabla \tilde v_j|_{\omega_X} \omega_X^n 
&\le& \Big(\int_{  U_{2\varepsilon}} |\nabla \tilde v_j|_{\omega_j }^2 e^{F_j} \omega_X^n \Big)^{1/2} \Big( \int_{  U_{2\varepsilon}}   e^{-F_j} \tr_{\omega_X} (\omega_j)  \omega_X^n \Big)^{1/2}\\
&\le& \delta_\varepsilon^{-1/2} \Big(\int_X |\nabla \tilde v_j|_{\omega_j }^2 e^{F_j} \omega_X^n \Big)^{1/2} \Big( \int_X   \omega_X^{n-1} \wedge \omega_j \Big)^{1/2} \\
&\le& (A\delta_\varepsilon)^{-1/2} \Big(\int_X |\nabla \tilde v_j|_{\omega_j }^2 e^{F_j} \omega_X^n \Big)^{1/2}  \to  0  
\end{eqnarray*}
by (\ref{smallc}) as $j \rightarrow \infty$. 
Therefore $\tilde v_j $ is uniformly bounded in $W^{1,1}(U_{2\varepsilon}, \omega_X)$ by the above estimate and (\ref{vjbd}).  By the Sobolev embedding theorem, after passing to a subsequence, we can assume that $\tilde v_j$ converges to $\tilde v_\infty$ in $L^1(  \overline{U_{3\varepsilon}}, \omega_X)$. Furthermore, since $\tilde v_j$ is uniformly bounded in $L^\infty(\overline{U_{3\varepsilon}})$ and converges almost everywhere to $\tilde v_\infty$ in $\overline{U_{3\varepsilon}}$, $\tilde v_\infty$ is also  bounded in $L^\infty(\overline{U_{3\varepsilon}}, \omega_X)$. 

Since $\lim_{j \rightarrow \infty} \int_{U_{2\varepsilon}} |\nabla \tilde v_j|_{\omega_X} \omega_X^n =0$, we have for any test function $f\in C_0^\infty(U_{2\varepsilon})$
\begin{eqnarray*}
\left| \int_X \tilde v_j  (\Delta f) \omega_X^n \right| 
&=&   \left| \int_X \langle \nabla \tilde v_j , \nabla f \rangle_{\omega_X} \omega_X^n \right| \\
&\leq& (\sup_X|\nabla f|_{\omega_X}) \int_{U_{2\varepsilon}}  |\nabla \tilde v_j |_{\omega_X} \omega_X^n \rightarrow  0
\end{eqnarray*}
as $j \rightarrow \infty$, where $\Delta$ is the Laplace operator with respect to $\omega_X$. Therefore by Weyl's lemma for the Laplace equation, $\tilde v_\infty$ solves the Laplace equation $\Delta \tilde v_\infty=0$ on $ U_{3\varepsilon}$   and $\tilde v_\infty \in C^\infty(U_{3\varepsilon})$.  For any smooth vector field $Y\in C_0^\infty(U_{3\varepsilon})$, we have 
 \bea\nonumber
 \int_X \langle \nabla \tilde v_\infty , Y\rangle_{\omega_X}\omega_X^n & = & - \int_X \tilde v_\infty \cdot {\mathrm{div}}_{\omega_X} Y \omega_X^n\\
 & = &- \lim_{j\to \infty} \int_X \tilde v_j \cdot {\mathrm{div}}_{\omega_X} Y \omega_X^n\nonumber\\
 & = &\lim_{j\rightarrow \infty} \int_X \langle \nabla \tilde v_j, Y\rangle_{\omega_X} \omega_X^n=0. \nonumber
 \eea
 Here the first and third lines follow from the divergence theorem, and the second equality holds since $\tilde v_j$ converge in $L^1 (\overline{U_{3\varepsilon}})$ to $\tilde v_\infty$ and the function ${\mathrm{div}}_{\omega_X} Y$ is compactly supported in $U_{3\varepsilon}$. By taking $Y = \eta \nabla\tilde v_\infty$ for any nonnegative function $\eta\in C^\infty_0(U_{2\delta})$, we see immediately that $\nabla \tilde v_\infty  \equiv 0$ on $U_{3\varepsilon}$, and hence $\tilde v_\infty$ is constant on $U_{3\varepsilon}$ since $U_{3\varepsilon}$ is connected.

 We first derive a uniform positive lower bound for  $\int_{U_{3\varepsilon}} |\tilde v_j| \omega_X^n$.  In fact, by the H\"older-Young inequality,  there exists $C>0$ such that for any $\delta>0$ and smooth functions $u$ and $F$, we have 
\bea\label{eqn:Young}|u| e^F = |\delta^{-1} u | e^{F + \log \delta}\le \delta e^{F} ( 1+ |F| + |\log \delta|) + C \delta^{-1} |u| e^{\delta^{-1} |u|}.\eea
Applying (\ref{eqn:Young}), we have 
\begin{eqnarray*}
1 &=& \frac{1}{[\omega_j]^n} \int_X |\tilde v_j | \omega_j^n = \int_X |\tilde v_j| e^{F_j}\omega_X^n \\
&\le  & \delta \int_{  X} e^{F_j} ( 1+ |F_j| + |\log \delta|) \omega_X^n + C \delta^{-1} \int_{  X} |\tilde v_j| e^{\delta^{-1} |\tilde v_j|} \omega_X^n  \\
&\le  & \delta \int_{  X} e^{F_j} ( 1+ |F_j| + |\log \delta|) \omega_X^n + C \delta^{-1} \int_{  U_{3\varepsilon}} |\tilde v_j| e^{\delta^{-1} |\tilde v_j|} \omega_X^n + Ce^{C\delta^{-1} }\int_{X\setminus U_{3\varepsilon}} \omega_X^n 
\\
&\le  & \delta \int_{ X} e^{F_j} ( 1+ |F_j| + |\log \delta|) \omega_X^n + C \delta^{-1} \int_{  U_{3\varepsilon}} |\tilde v_j| e^{\delta^{-1} |\tilde v_j|} \omega_X^n + 2 \varepsilon Ce^{C\delta^{-1} } \\
&\le   & \frac{1}{2} + C_\delta ([\omega_X]^n)^{-1}\int_{  U_{3\varepsilon}}  |\tilde v_j| \omega_X^n ,
\end{eqnarray*}
for all $j$ and some uniform constants $C=C(A, p, K)$ and $C_\delta=C_\delta(A, p, K)>0$, if we choose $\delta=\delta(A, p, K)>0$ sufficiently small such that 
$$ \delta  \int_{ X} e^{F_j} ( 1+ |F_j| + |\log \delta|) \omega_X^n< \frac{1}{4}, $$
and then choose %
\begin{equation}\label{ep1}
\varepsilon < \varepsilon_1= \frac{e^{-C\delta^{-1} }}{8C} .
\end{equation}
 Immediately we have  %
\begin{equation}\label{lowerbdv}
([\omega_X]^n)^{-1} \int_{U_{3\varepsilon}}  |\tilde v_j| \omega_X^n\ge \left( 2C_\delta\right)^{-1}
\end{equation}
 for sufficiently large $j$.

On the other hand,  we can extend $\tilde v_\infty$ to a constant function on $X$. By applying the H\"older-Young inequality again and by the fact that on $U_{2\varepsilon}$ is connected, for any $\epsilon'>0$, there exists $C_1>0$ 
and $C_2=C_2(A, p, K)>0$ 
such that for sufficiently large $j$, we have
\begin{eqnarray*}
|\tilde v_\infty|   
&=& \left| \frac {1}{[\omega_j]^n} \int_{ X} \left( \tilde v_\infty - \tilde v_j\right) \omega_j^n \right| \\
& \leq & \frac {1}{[\omega_j]^n} \int_{U_{3\varepsilon}} \left| \tilde v_\infty - \tilde v_j \right| \omega_j^n   + \frac{1}{[\omega_j]^n}\int_{X\setminus U_{3\varepsilon}}   |v_\infty-\tilde v_j|   \omega_j^n\\
&= &   \int_{U_{3\varepsilon}} \left| \tilde v_\infty - \tilde v_j\right| e^{F_j}\omega_{X}^n +\int_{X\setminus U_{3\varepsilon}}   |v_\infty-\tilde v_j|   e^{F_j}\omega_X^n \\
&\le & \epsilon' \int_{X} e^{F_j} (1+ |F_j| + |\log \epsilon' |) \omega_X^n + C_1 (\epsilon')^{-1}e^{(\epsilon')^{-1}\sup_X |\tilde v_j-\tilde v_\infty|} \int_{ U_{3\varepsilon}}  |\tilde v_j - \tilde v_\infty|  \omega_X^n \\
&&  + C_1(\epsilon')^{-1} e^{(\epsilon')^{-1}  \sup_X  \left( |v_\infty-\tilde v_j| \right) } \int_{X \setminus U_{3\varepsilon}}    |v_\infty-\tilde v_j|   \omega_X^n \\
&\le &  C_2  (\epsilon')^{1/2} + e^{C_2 ( \epsilon')^{-1}} \varepsilon \\
&<& (4C_\delta)^{-1},
\end{eqnarray*}
if we choose $\epsilon'$  and $\varepsilon$ with 
\begin{equation}\label{ep2}
\epsilon'  < (8C_2C_\delta)^{-2}  , \text{ and }\varepsilon <  \varepsilon_2=\left(8  e^{C_2 ( \epsilon')^{-1}}  C_\delta\right)^{-1} , 
\end{equation}

Since $\tilde v_j$ converges to $\tilde v_\infty$ in $L^1( \overline{U_{3\varepsilon}})$,  for sufficiently large $j$, we have 
$$ ([\omega_X]^n)^{-1}\int_{  U_{3\varepsilon}}  |\tilde v_j| \omega_X^n <  ([\omega_X]^n)^{-1} \int_{  U_{3\varepsilon}}  |\tilde v_\infty| \omega_X^n + (4C_\delta)^{-1}<  (2C_\delta)^{-1}.$$
This contradicts the lower bound (\ref{lowerbdv}).
From now on, we will fix the choice for 
$$\varepsilon=\frac{\min(\varepsilon_1, \varepsilon_2)}{2}$$ from (\ref{ep1}) and (\ref{ep2}) for the parameter $\varepsilon$ in the assumption of the lemma. 

We have now completed the proof of the lemma.
\end{proof}

Let $\omega$ be a K\"ahler metric on $X$.  We let $G(x,\cdot)$ be the Green's function of $(X,\omega)$ with base point $x$, for any $x \in X$.

\begin{lemma}\label{lemma Green} There exists $\varepsilon=\varepsilon(X, \omega_X, n, A, p, K)>0$ such that if 
\begin{enumerate}

\item $\gamma\geq 0$ is a continuous function on $X$ such that 
$$|\{\gamma=0\}|_{\omega_X} \leq \varepsilon, ~ \{\gamma >0 \}~ \text{is connected},$$

\item $\omega\in \mathcal{W}(X, \omega_X, n, A, p, K; \gamma)$,

\end{enumerate} 
then there exists   $C=C(X, \omega_X, n, A, p, K, \varepsilon, \gamma)>0$ such that for any $x\in X$
$$\int_X |G(x,\cdot)| \omega^n \le C,\, \mbox{ and } \inf_{y\in X} G(x, y) \ge - \frac{C}{[\omega]^n},$$
where $G(x, \cdot)$ is the Green's function of $(X, \omega)$. 
\end{lemma}

\begin{proof}

We now fix $\omega\in \mathcal{W}(X, \omega_X, n, A, p, K; \gamma)$ satisfying the assumption and $x\in X$. Take a sequence of smooth functions $h_k$ which are uniformly bounded and  converge  in $L^q(X, \omega)$ for some fixed sufficiently large $q>0$, to $-\chi_{\{G(x,\cdot) \le 0\}} + \frac{1}{[\omega]^n}\int_{\{G(x,\cdot) \le 0\}} \omega^n$, where  we denote $\chi_E$ to be the characteristic function of a Borel set $E$. We can also choose $h_k$ to satisfy 
$$\sup_X |h_k|\le 2,\quad\mbox{and }\quad \frac{1}{[\omega]^n} \int_X h_k \omega^n = 0.$$
Immediately, there exists a unique smooth solution  solving  the linear equation
$$\Delta_{\omega} v_k = h_k, \quad \frac{1}{[\omega]^n} \int_X v_k \omega^n = 0.$$
By Lemma \ref{lemma 6}, there exists $C>0$ independent of $k$ such that  
$$\sup_X |v_k|\le C .$$
 Applying the Green's formula, we have by the dominated convergence theorem
$$v_k(x) = \int_X G(x,y) (-h_k(y)) \omega^n(y) \to \int_{\{G(x,\cdot)\le 0\}} G(x,\cdot) \omega^n$$
as $k\rightarrow \infty$.  Combining this with the fact that 
$$\int_{\{G(x,\cdot)\ge  0\}} G(x,\cdot) \omega^n = -  \int_{\{G(x,\cdot)\le 0\}} G (x,\cdot) \omega^n.$$
we easily find that $\int_X |G(x,\cdot)|\omega^n\le C$.

\smallskip

For the lower bound of the Green's function, we apply Lemma \ref{lemma key} to the function $v := - [\omega]^n \cdot G(x,\cdot)$ and $a = 1$. It then follows that
$$- [\omega]^n \cdot \inf_X G (x,\cdot)\le C ( [\omega]^n + \int_X |G(x,\cdot)| \omega ^n  )\le C.$$
This completes the proof of the lemma. 
\end{proof}

We observe that Lemma \ref{lemma Green} implies a lower bound of the first nonzero eigenvalue of the Laplacian operator $\Delta_{\omega}$. To see this, suppose $\lambda_1>0$ is such an eigenvalue and $f\in C^\infty(X)$ is an associated eigenfunction normalized by $\int_X f^2 \omega^n = [\omega]^n$. Then  $\Delta_{\omega} f = - \lambda_1 f$. If we let $x_0\in X$ be a maximum point of $|f|$, by the Green's formula we have
$$0\neq f(x_0) = \frac{1}{[\omega]^n}\int_X f \omega^n -  \int_X G(x_0, \cdot) \Delta_{\omega} f \omega_t^n = \lambda_1 \int_X G(x_0, \cdot) f \omega^n.$$
Hence
$$|f(x_0)|\le \lambda_1 |f(x_0)|\int_X |G(x_0,\cdot)| \omega^n\le C |f(x_0)| \lambda_1,$$
by Lemma \ref{lemma Green}.
This immediately gives the uniform positive lower bound of $\lambda_1$.

\medskip

For convenience of notation, we write
\begin{equation}\label{eqn:G t}{\mathcal G} (x,\cdot) = G (x,\cdot) - \inf_{x, y\in X} G(x, y) +1 >0.\end{equation}
It is clear that $\int_X {\mathcal G} (x,\cdot)\omega^n \le C$.

\begin{lemma}\label{lemma 65}

There exist $\varepsilon=\varepsilon(X, \omega_X, n, A, p, K)>0$ and $\varepsilon'=\varepsilon'(n, p)>0$ such that if 
\begin{enumerate}

\item $\gamma\geq 0$ is a continuous function on $X$ such that 
$$|\{\gamma=0\}|_{\omega_X} \leq \varepsilon, ~ \{\gamma >0 \}~ \text{is  connected},$$

\item $\omega\in \mathcal{W}(X, \omega_X, n, A, p, K; \gamma)$,

\end{enumerate} 
then there exists $C=C(A, p, K, \varepsilon, \gamma, \varepsilon')>0$ such that for any $x\in X$, we have
$$\int_X {\mathcal G}(x,\cdot)^{1+\varepsilon'} \omega^n \le C ([\omega]^n)^{-\varepsilon'}.$$

\end{lemma}

\begin{proof}

 We fix $x\in X$ and a small constant $\varepsilon'>0$ to be determined.
Fix a large $k\gg1$ and consider a smooth positive function $H_k$ which is a smoothing of $ \min\{{\mathcal G}(x,\cdot), k\}$. Without loss of generality, we can assume that $H_k$  converges  increasingly to ${\mathcal G}(x,\cdot)$ as $k\to \infty$. 
In particular,  there exists $C=C(A, p, K, \varepsilon, \gamma)>0$ such that  for any $k$
$$0<\int_X H_k e^{F}\omega_X^n\le \frac{1}{[\omega]^n} \int_{X}{\mathcal G}(x,\cdot) \omega^n \le \frac{C}{[\omega]^n},$$
where $F= ([\omega]^n)^{-1} \frac{\omega^n}{\omega_X^n}.$

We now consider the following linear equation on $X$  
\begin{equation}\label{eqn:new key 1}
\left\{\begin{array}{ll}
&\Delta_{\omega} u_k = - H_k^{\varepsilon'} + \frac{1}{[\omega]^n} \int_X H_k^{\varepsilon'} \omega^n,\\
&\frac{1}{[\omega]^n} \int_X u_k\omega^n = 0.\end{array}\right.
\end{equation}
Equation (\ref{eqn:new key 1}) admits a unique smooth solution since the smooth function on the right-hand side of the first equation has integral $0$. We cannot apply the maximum principle to $u_k$ directly since the term $- H_k^{\varepsilon'}$ on the right-hand side of \ref{eqn:new key 1} is unbounded. 

We will let $\chi\in [\omega]$ be a smooth closed $(1,1)$-form such that $ \|\chi \|_{C^3(X, \omega_3)}$ is uniformly bounded some constant that only depends on $A$. We let 
$$\omega=\chi + \ddbar \varphi, ~\sup_X \varphi =0, $$ and let $$\hat H_k: = [\omega]^n \cdot H_k.$$
Then we  consider the following auxiliary complex Monge-Amp\`ere equation which admits a smooth solution by \cite{Y}
\begin{equation}\label{eqn:aMA new}
\frac{1}{[\omega]^n} (\chi  + \ddbar \psi_k)^n= \frac{(\hat H_k)^{n\varepsilon'} + 1}{  \int_X \left((\hat H_k)^{n\varepsilon'} + 1 \right) \omega^n} \omega^n =  \frac{ (\hat H_k)^{n\varepsilon'} + 1}{B_k} e^{F} \omega_X^n,
\end{equation}
with 
$$\sup_X \psi_k = 0,~ B_k =  \int_X ( (\hat H_k)^{n\varepsilon'} + 1) e^{F} \omega_X^n. $$ 
By the H\"older inequality, there exists $C=C(A, p, K, \varepsilon', \gamma)>0$ for sufficiently small $0<\varepsilon'<n^{-1}$ such that for all $k$, we have
\begin{equation}\label{eqn:Bk}
C^{-1}\le \int_X \gamma \omega_X^n \le B_k\le \int_X e^F\omega_X^n + \Big( \int_X e^F \omega_X^n \Big)^{1-n\varepsilon'} \Big( \int_X \hat H_k e^{F} \omega_X^n \Big)^{n\varepsilon'}\le C.
\end{equation}
 For fixed $p'\in (n, p)$, the $p'$-th entropy of  the function $((\hat H_k)^{n\varepsilon'} + 1 )e^{F}/B_k$ on the right-hand side of (\ref{eqn:aMA new}) satisfies
\begin{eqnarray}\label{eqn:newest 1}
&& \frac{1}{B_k}\int_X ( (\hat H_k)^{n\varepsilon'} + 1  )  \Big| -\log B_k + F + \log ( 1+ (\hat H_k) ^{n\varepsilon'})   \Big|^{p'} e^{F} \omega_X^n\\
 &\le & \frac{|\log B_k|^{p'}}{B_k} \int_X ( (\hat H_k)^{n\varepsilon'} + 1  ) e^{F} \omega_X^n + \frac{1}{B_k} \int_X ( (\hat H_k)^{n\varepsilon'} + 1  ) \left( \log( (\hat H_k)^{n\varepsilon'} + 1  )\right)^{p'} e^{F} \omega_X^n \nonumber \\
\nonumber & & \quad  + \frac{1}{B_k} \int_X ( (\hat H_k)^{n\varepsilon'} + 1  ) |F|^{p'} e^{F} \omega_X^n. \nonumber 
\end{eqnarray}

The first integral on the right hand side in (\ref{eqn:newest 1}) is bounded due to the estimate of the constant $B_k$ in \ref{eqn:Bk}, the H\"older inequality and the uniform $L^1$-bound of %
$$\int_X \hat H_k e^{F}\omega_X^n \leq \int_X \mathcal{G}(x, \cdot)\omega^n$$
by Lemma \ref{lemma Green}. 

The second integral on the right hand side in (\ref{eqn:newest 1}) is also uniformly bounded by a similar argument since  
$$ \frac{1}{B_k} \int_X (\hat H_k^{n\varepsilon'} + 1  ) [\log(\hat H_k^{n\varepsilon'} + 1  )]^{p'} e^{F} \omega_X^n\le C\int_X \hat H_k^{\varepsilon'} (\hat H_k^{n\varepsilon'} + 1) e^{F}\omega_X^n \le C ,$$
by the H\"older inequality and the calculus inequality $(\log (1+x))^{p'}\le C x^{\varepsilon'/n}$ for any $x>0$. We have also chosen $\varepsilon'>0$ small so that $(n+1)\varepsilon'<1$. 

To deal with the last integral in (\ref{eqn:newest 1}), we observe that by Young's inequality
$$( (\hat H_k)^{n\varepsilon'} + 1  ) |F|^{p'}\le \frac{|F|^p}{p/p'} + \frac{(\hat H_k^{n\varepsilon'} + 1  )^{(p/p')^*}}{(p/p')^*},$$
where $(p/p')^*>1$ is the conjugate exponent of $p/p'>1$. Hence the last term in (\ref{eqn:newest 1}) satisfies
\begin{eqnarray*}
&&\frac{1}{B_k} \int_X ( (\hat H_k)^{n\varepsilon'} + 1  ) |F|^{p'} e^{F} \omega_X^n \\
&\le& C\int_X |F|^pe^{F} \omega_X^n + C \int_X ( (\hat H_k)^{n \varepsilon' (p/p')^*} + 1  )e^{F} \omega_X^n\\
&\le& C,
\end{eqnarray*}
if we choose $\varepsilon'$ small so that $n \varepsilon' (p/p')^*<1$. 

From now on we fix a small $\varepsilon'>0$ that meets the requirements above and so the $p'$-th entropy of  the function on the right-hand side of (\ref{eqn:aMA new}) is uniformly bounded. We apply Corollary \ref{linf2} to conclude that 
\begin{equation}\label{c0bds3}
\sup_X |\psi_k - \varphi|\le C,
\end{equation}
for some uniform constant $C=C(A, p, K, \varepsilon, \gamma, \varepsilon')>0$.

 We now consider the function
\begin{equation}\label{eqn:def v}
v_k: = (\psi_k - \varphi) - \frac{1}{ [\omega]^n} \int_X (\psi_k - \varphi) \omega^n + \varepsilon'' u_k ,
\end{equation}
where $\varepsilon''>0$ is a suitable constant to be chosen later. It follows from the definition that $\frac{1}{[\omega]^n}\int_X v_k \omega^n = 0$ and $v_k$ is a smooth function. 

Let $\omega_{\psi_k} = \chi + \ddbar \psi_k$. We then calculate the Laplacian of $v$ in (\ref{eqn:def v}) and there exists $C>0$ such that 
\bea
\Delta_{\omega} v_k \nonumber &= & \tr_{\omega} (\omega_{\psi_k}) - n + \varepsilon''  \Delta_{\omega} u_k\\
\nonumber &\ge & n \Big( \frac{\omega_{\psi_k}^n}{\omega^n} \Big)^{1/n} - n - \varepsilon'' H_k^{\varepsilon'} + \frac{\varepsilon''}{[\omega]^n} \int_X H_k^{\varepsilon'} \omega^n\\
\nonumber &= & n B_k^{-1/n} ( \hat H_k^{n\varepsilon'} + 1 )^{1/n} - n - \varepsilon'' H_k^{\varepsilon'} + \frac{\varepsilon''}{([\omega]^n)} \int_X H_k^{\varepsilon'} \omega^n\\
\nonumber &\ge & n C^{-1}  ([\omega]^n)^{\varepsilon'} H_k^{\varepsilon'} - n - \varepsilon'' H_k^{\varepsilon'} \ge -n,
\eea
if we choose $\varepsilon'' = n C^{-1} ([\omega]^n)^{\varepsilon'}$. We apply the Green's formula to the function $v_k$ at $x$
\bea
v_k(x) & = & \frac{1}{[\omega]^n} \int_X v_k \omega ^n + \int_X G (x,\cdot) (-\Delta_{\omega} v_k) \omega^n =  \int_X {\mathcal G}(x,\cdot) (-\Delta_{\omega} v_k) \omega^n \nonumber\\
&\le & n \int_{X} {\mathcal G}(x,\cdot) \omega^n\le C,\nonumber
\eea
where the last inequality follows from the uniform $L^1(X,\omega^n)$-bound of ${\mathcal G}(x,\cdot)$. It then follows from (\ref{c0bds3})  that 
$$u_k(x)\le C ([\omega]^n)^{-\varepsilon'}$$%
 for a uniform constant $C>0$. 
 We now apply the Green's formula to $u_k$ at $x\in X$
\bea\nonumber
u_k(x) & = &\frac{1}{[\omega]^n} \int_X u_k\omega^n + \int_X {\mathcal G}(x,\cdot) (-\Delta_{\omega} u_k) \omega^n\\
& = & \nonumber \int_X {\mathcal G}(x,\cdot) \Big( (H_k)^{\varepsilon'} - \frac{1}{[\omega]^n} \int_X (H_k)^{\varepsilon'} \omega^n   \Big) \omega^n.
\eea
It then follows that 
\begin{eqnarray*}
\int_X {\mathcal G}(x,\cdot) (H_k)^{\varepsilon'} \omega^n &\le& u_k(x) + C \frac{1}{[\omega]^n} \int_X (H_k)^{\varepsilon'} \omega^n \\
&\le& C ([\omega]^n)^{-\varepsilon'} + C \Big(\frac{1}{ [\omega]^n} \int_X H_k \omega^n \Big)^{\varepsilon'} \\
&\le& 2C ([\omega]^n)^{-\varepsilon'}  ,
\end{eqnarray*}
for some uniform constant $C>0$.
Letting $k\to\infty$ and applying the monotone convergence theorem, we can conclude that
$$\int_X {\mathcal G}(x,\cdot)^{1+\varepsilon'} \omega^n \le C ([\omega]^n)^{-\varepsilon'},$$
for some uniform constant $C>0$. The proof of the lemma is now complete.
\end{proof}

We observe the following elementary estimate which follows easily from the Green's formula.
\begin{lemma}\label{lemma gradient} Under the same assumptions of Lemma \ref{lemma 65}, 
for any $\beta>0$ we have
\begin{equation}\label{eqn:grad 1}
\sup_{x\in X} \int_X \frac{ |\nabla_y {\mathcal G}(x,y)|_{\omega(y)}^2   }{{\mathcal G}(x,y)^{1+\beta}} \omega^n(y) \le 
\frac{([\omega]^n)^\beta}{\beta}.
\end{equation} 
\end{lemma}
\begin{proof} Fix $\beta>0$ and a point $x\in X$. 
The function $u(y): = {\mathcal G}(x,y)^{-\beta}$ is a continuous function with $u(x) = 0$ and $u\in C^\infty(X\backslash \{x\})$. By the definition of ${\mathcal G}$ in (\ref{eqn:G t}), for any $y\in X$, we have 
\begin{equation}\label{eqn:u bound}0\le u(y) \le ([\omega]^n)^{\beta}.\end{equation}
Applying the Green's formula, we have
\bea
0= u(x) \nonumber &= & \frac{1}{[\omega]^n} \int_X u \omega^n + \int_X {\mathcal G}(x, \cdot) (-\Delta_{\omega} u) \omega^n\\
\nonumber &= & \frac{1}{[\omega]^n} \int_X u \omega^n -\beta \int_X \frac{|\nabla {\mathcal G}(x,\cdot)|^2_{\omega}}{ {\mathcal G}(x,\cdot)^{1+\beta}} \omega^n.
\eea
In the last inequality,  we apply the integration by parts using the asymptotic behavior of ${\mathcal G}(x,y)$ near $y=x$. The lemma then follows easily from \ref{eqn:u bound}.
\end{proof}

Finally we are ready to derive the uniform $L^1(X,\omega^n)$ bound for the gradient of $G(x,\cdot)$. 

\begin{lemma}\label{lemma gradient final} Under the same assumptions of Lemma \ref{lemma 65}, for any $s\in [1, \frac{2+2\varepsilon'}{2+\varepsilon'})$
there is a uniform constant $C=C(s)>0$ such that for any $x\in X$, we have
 \begin{equation}
\int_X |\nabla G(x,\cdot)|^s_{\omega} \omega^n \le \frac{C}{([\omega]^n)^{s-1}}.
\end{equation}
\end{lemma}
\begin{proof} It suffices to prove the same estimate for ${\mathcal G}(x,\cdot)$. For fixed $x\in X$, we regard ${\mathcal G}(y) : = {\mathcal G}(x,y)$ as a function of $y$. By fixing $s\in [1, \frac{2+2\varepsilon'}{2+\varepsilon'})$ and applying H\"older inequality,  we have
\bea\label{gres}
\int_X |\nabla G(x,\cdot)|^s_{\omega} \omega^n &\le & \Big(\int_X \frac{ | \nabla {\mathcal G}|^2_{\omega}}{{\mathcal G}^{1+\beta}} \omega^n \Big)^{s/2}  \Big( \int_X {\mathcal G}^{1+\varepsilon'} \omega^n\Big)^{(2-s)/2}\\
&\le &\nonumber C ([\omega]^n)^{\beta s/2}  ([\omega]^n)^{-\varepsilon' (2-s)/2} \\
&=& C ([\omega]^n)^{1-s}, \nonumber
\eea
where $\beta>0$ is chosen by $1+\beta = (1+\varepsilon')\frac{2-s}{s}$.  We apply the estimates in Lemmas \ref{lemma 65} and \ref{lemma gradient} for the second inequality in (\ref{gres}). 
\end{proof}

Combining the estimates above, we have established the following main result of this section.

\begin{proposition}  \label{gres!}  For any $A, K>0$ and $p>n$, there exist $\varepsilon=\varepsilon(X, \omega_X, n, A, p, K)>0$ and $\varepsilon'=\varepsilon'(n, p)>0$ such that if 
\begin{enumerate}

\item $\gamma\geq 0$ is a continuous function on $X$ such that 
$$|\{\gamma=0\}|_{\omega_X} \leq \varepsilon, ~ \{\gamma >0 \}~\text{ is connected},$$

\item $\omega\in \mathcal{W}(X, \omega_X, n, A, p, K; \gamma)$,

\item $s\in [1, \frac{2+2\varepsilon'}{2+\varepsilon'})$, 

\end{enumerate} 
 there exist $C_1=C(X, \omega_X, n, A, p, K, \gamma, \varepsilon)>0$, $C_2=C_2(X, \omega_X, n, A, p, K, \gamma, \varepsilon, \varepsilon')>0$ and $C_3=C_3(X, \omega_X, n, A, p, K, \gamma, \varepsilon, \varepsilon', s)>0$ such that   
\begin{eqnarray*}
&& \inf_{y\in X} G(x, y) \geq  - C_1 \left( [\omega]^n\right)^{-1},\\
&& \int_X |G(x, \cdot)|^{1+\varepsilon'} \omega^n \leq  C_2,  \\
&& \int_X |\nabla  G(x, \cdot)|^{1+s} \omega^n \leq C_3,
\end{eqnarray*}
for any $x\in X$.

\end{proposition}


\section{Diameter and volume estimates}   \label{secdiam}

\setcounter{equation}{0}

In this section, we will establish the following diameter and volume estimates by applying Proposition \ref{gres!}.

\begin{proposition} \label{diaest}   For any $A, K>0$ and $p>n$, there exist $\varepsilon=\varepsilon(X, \omega_X, n, A, p, K)>0$ and $\varepsilon'=\varepsilon'(n, p)>0$ such that if 
\begin{enumerate}

\item $\gamma\geq 0$ is a continuous function on $X$ such that 
$$|\{\gamma=0\}|_{\omega_X} \leq \varepsilon, ~ \{\gamma >0 \}~\text{ is connected},$$

\item $\omega\in \mathcal{W}(X, \omega_X, n, A, p, K; \gamma)$,
 
\end{enumerate} 
 there exist $\alpha=\alpha(n, p)$, $C=C(X, \omega_X, n, A, p, K, \gamma, \varepsilon)>0$ and $c=c(X, \omega_X, n, A, p, K, \gamma, \varepsilon, \alpha)>0$ such that
\begin{eqnarray}\label{diavol}%
&&{\rm diam}(X, \omega) \leq C, \\
&& \frac{{\mathrm{Vol}}_{\omega} (B_{\omega}(x, R))}{[\omega]^n}\ge c R^\alpha,
\end{eqnarray}
for any $x\in X$ and  $R\in (0,1]$.

\end{proposition}

\begin{proof}

We first prove the diameter bound. Since $(X, \omega)$ is compact and complete,  there exist a pair of points $x_0, y_0\in X$ such that $d_{\omega}(x_0, y_0) = {\mathrm{diam}}(X,\omega)$. We define the $1$-Lipschitz function $d(\cdot)$ on $X$ by
$$d(y) = d_{\omega}(x_0, y) .$$
 Apply the Green's formula to $d$ at a point $x\in X$. We obtain
\begin{equation}\label{eqn:d t}
d(x) = \frac{1}{[\omega]^n} \int_X d(y) \omega^n(y) + \int_{X} \langle{ \nabla_y  G}(x, y), \nabla d(y)\rangle_{\omega(y)} \omega^n(y). 
\end{equation}
By letting  $x = x_0$, we have $d(x_0) = 0$ and
\begin{eqnarray*}
\frac{1}{[\omega]^n} \int_X d(y) \omega^n(y) &=& - \int_{X} \langle{ \nabla_y  G}(x_0, y), \nabla d(y)\rangle_{\omega(y)} \omega^n(y)\\
&\le& \int_X |\nabla_y G(x_0,y)|_{\omega(y)}\omega^n(y).
\end{eqnarray*}
Finally we establish the uniform diameter by applying \ref{eqn:d t} to $x=z_0$ with
\begin{eqnarray*}
 {\mathrm{diam}}(X,\omega) 
& = & d(y_0)\\
& = & \frac{1}{[\omega]^n} \int_X d(y) \omega^n(y) + \int_{X} \langle{ \nabla_y  G}(y_0, y), \nabla d(y)\rangle_{\omega(y)} \omega^n(y)\\
&\le & \int_X |\nabla_y G(x_0,y)|_{\omega(y)} \omega^n(y)+ \int_X |\nabla_y G (y_0,y)|_{\omega(y)} \omega^n(y) \\
&\le& C, 
\end{eqnarray*}
for some uniform constant $C>0$ by Proposition \ref{gres!}. 

\medskip

We now turn to the non-collapsing estimate for metric balls in $(X,\omega)$. Fix a point $x\in X$ and a number $R \in (0,1]$. Let $B(x,R)\subset X$ be the geodesic ball in $(X,\omega)$ with center $x$ and radius $R>0$. We choose a smooth cut-off function $\eta$ with support in $B(x, R)$ satisfying 
$$\eta\equiv 1, ~{\rm on}~B\Big(x,\frac{R}{2} \Big),  ~\sup_X|\nabla \eta|_{\omega}\le \frac{4}{R}.$$
 We let $d(y) = d_{\omega}(x, y)$ be the geodesic distance from $x$ to $y\in X$. Applying the Green's formula to the Lipschitz function $d\cdot\eta$, we have for any $z\in X$
\begin{equation}\label{eqn:non 1}
d(z) \eta(z) = \frac{1}{[\omega]^n} \int_X d(y) \eta(y) \omega^n(y) + \int_{X} \langle\nabla_y G(z,y), \eta(y) \nabla d(y) + d(y) \nabla \eta(y)    \rangle_{\omega(y)} \omega^n(y).
\end{equation}
Take $s = \frac{2+1.5\varepsilon'}{2+\varepsilon'}>1$ for  $\varepsilon'$  from the assumption in Proposition \ref{gres!}. We apply (\ref{eqn:non 1}) to a point $\tilde z\in X\backslash {\overline{B(x,R)}}$. Then $d(\tilde z) \eta(\tilde z) = 0$ and by Lemma \ref{lemma gradient final}, we have
\begin{eqnarray*}
&& \frac{1}{[\omega]^n} \int_X d(y) \eta(y) \omega^n(y) \\
&\le & \int_{X} |\nabla_y G(\tilde z,y)|_{\omega(y)}  \left( \eta(y) + d(y) |\nabla \eta(y)|_{_{\omega(y)}} \right) \omega^n(y)\\
&\le & 5\nonumber \Big( \int_{X} |\nabla_y G(\tilde z,y)|^s_{\omega(y)}\omega^n(y)\Big)^{1/s} \cdot \Big({\mathrm{Vol}}_{\omega}(B(x, R))\Big)^{1/s^*}\\
&\le &\label{eqn:non 2} C ([\omega]^n)^{- \frac{s-1}{s}}\Big({\mathrm{Vol}}_{\omega}(B(x, R))\Big)^{1/s^*}\\
&=&C \Big(\frac{{\mathrm{Vol}}_{\omega}(B(x, R))}{[\omega]^n}\Big)^{1/s^*},
\end{eqnarray*}
where $s^* = \frac{s}{s-1}$ is the conjugate exponent of $s$. Next we apply (\ref{eqn:non 1}) to a point $\hat z \in \partial B(x, R/2)$ where $d(\hat z) \eta(\hat z) = R/2$. Applying the above estimate along with the same argument, we have
\begin{eqnarray*}
\frac {R}{2} &\le &  \frac{1}{[\omega]^n} \int_X d(y) \eta(y) \omega^n(y) +  \int_{X} |\nabla_y G(\hat z,y)|_{\omega(y)} \left( \eta(y) + d(y) |\nabla \eta(y)|_{_{\omega(y)}}\right) \omega^n(y)\\
&\le &\nonumber C \left(\frac{{\mathrm{Vol}}_{\omega}(B(x, R))}{[\omega]^n}\right)^{1/s^*},
\end{eqnarray*}
for some uniform constant $C>0$. This immediately gives a lower bound of the volume of $B(x, R)$,
$$\frac{{\mathrm{Vol}}_{\omega}(B(x, R))}{ [\omega]^n} \ge c R^{\alpha},$$
for some uniform constants $\alpha = s^*(n, p)>0$ and $c=c(A, p, K, \gamma, \varepsilon, \varepsilon', \alpha)>0$. 
\end{proof}

\noindent{\bf Remark.} We briefly explain an application of the noncollapsing estimate (c) of Theorem \ref{thm:main1} to the pre-compactness in Gromov-Hausdorff (GH) topology. Let $(X, \omega_{j})$ be a sequence of K\"ahler metrics satisfying the assumptions in Theorem \ref{thm:main1}. By Gromov's precompactness theorem, it suffices to verify the following: 

\smallskip

{\em for any $\epsilon>0$, there exists an $N(\epsilon)>0$ which is independent of $j$ such that there exists an $\epsilon$-dense set $\{x_j^a\}_{a=1}^{M_j}$ in the metric space $(X,\omega_{j})$ with $M_j\le N(\epsilon)$.}

\medskip

In fact, suppose $\{x_j^a\}_{a=1}^{M_j}$ is an $\epsilon$-dense set in the metric space $(X,\omega_{j})$, by which we mean a maximal collection of points where any two of them have distance at least $\epsilon$. By definition, the geodesic balls $\{  B_{j}(x_j^a, \epsilon/2)  \}_a$ are pairwise disjoint, hence by (c) of Theorem \ref{thm:main1}, 
$$c (\epsilon/2)^{\alpha} ([\omega]^n) M_j\le \sum_{a=1}^{M_j} {\mathrm{Vol}}_{\omega_{j}} ( B_{\omega_{j}} (x_j^a, \epsilon/2)  ) \le ([\omega]^n) = {\mathrm{Vol}}(X,\omega_t^n),$$
so $M_j\le c^{-1} (\epsilon/2)^{-\alpha}=:N(\epsilon)$. 

This shows that up to a subsequence the metric spaces $(X,\omega_{j})$ converge in GH topology to a {\em compact} metric space $(Z,d_Z)$.

\medskip

Now we can complete the proof of Theorem \ref{thm:main1}.

\noindent {\it Proof of Theorem \ref{thm:main1}.} It suffices to show that if $S$ is a closed subset of $X$ with $\dim_{\mathcal{M}} S < 2n-1$, then $X\setminus S$ is connected. Since the Cech cohomological dimension is always no greater than the topological dimension, which is no greater than Minkowski dimension, we have
$$\check H^{2n-1}(S) =\check H^{2n}(S)= 0$$
and so by Poincare-Alexander-Lefschetz duality $H_1(X, X\setminus S) = \check H^{2n-1}(S) = 0$ and $H_0(X, X\setminus S) = \check H^{2n}(S) = 0$ (c.f. Theorem 8.3, Chapter VI, in \cite{Br}). The exact sequence for reduced coholomology gives
$$0=H_1(X, X\setminus S) \rightarrow \tilde H_0(X\setminus S) \rightarrow \tilde H_0(X)\rightarrow  H_0(X,X\backslash S) = 0. $$
Therefore $\tilde H_0(X\setminus S) = \tilde H_0(X) = {\mathbb Z}$ and so $X\setminus S$ is connected.  Then Theorem \ref{thm:main1} is a direct consequence of  Proposition \ref{gres!} and Proposition \ref{diaest}. \qed


\section{A uniform Sobolev inequality}

In this section, we will prove a special Sobolev-type inequality for K\"ahler metrics satisfying the assumption in Proposition \ref{gres!}. The main feature of this inequality is the uniformity of the constants.

We first improve Lemma \ref{lemma 6} in the following lemma.

\begin{lemma} \label{7sol} For any $A, K>0$ and $p>n$, there exist $\varepsilon=\varepsilon(X, \omega_X, n, A, p, K)>0$ and $\varepsilon'=\varepsilon'(n, p)>0$ such that if 
\begin{enumerate}

\item $\gamma\geq 0$ is a continuous function on $X$ such that 
$$|\{\gamma=0\}|_{\omega_X} \leq \varepsilon, ~ \{\gamma >0 \}~ \text{is connected},$$

\item $\omega\in \mathcal{W}(X, \omega_X, n, A, p, K; \gamma)$,

\end{enumerate} 
then there exists $C=C(A, p, K, \varepsilon, \gamma, \varepsilon')>0$ such that for any $v\in C^\infty(X)$ satisfying
$$\frac{1}{[\omega]^n}  \int_X |\Delta v |^{(1+\varepsilon')^*} \omega^n \leq 1, ~ \int_X v \omega^n =0, $$
we have
$$\sup_X |v| \leq C. $$

\end{lemma}

\begin{proof} Let $p = 1+ \varepsilon'$ and $q=p^*$. Then applying Proposition \ref{gres!}, there exists $C=C(A, p, K, \varepsilon, \gamma, \varepsilon')>0$ such that 
\begin{eqnarray*}
v(x)  &=&  - \int_X G(x, y) \Delta v(y) \omega^n(y) \\
&\leq& \left( \int_X |G(x, \cdot)|^p \omega^n \right)^{1/p} \left( \int_X |\Delta v|^q \omega^n \right)^{1/q}\\
&\leq& C \left( [\omega]^n \right)^{-\varepsilon'/p} \left( \int_X |\Delta v|^q \omega^n \right)^{1/q}\\
&\leq& C \left( [\omega]^n \right)^{-\varepsilon'/p+1/q}\\
&\leq & C,
\end{eqnarray*}
since $-\frac{\varepsilon'}{p} + \frac{1}{q} = 1- \frac{\varepsilon'+1}{p}=0$.
\end{proof}

We can now apply Lemma \ref{7sol} to derive a Sobolev-type inequality with large exponents.

\begin{lemma} \label{sobol} For any $A, K>0$ and $p>n$, there exist $\varepsilon=\varepsilon(X, \omega_X, n, A, p, K)>0$ and $\varepsilon'=\varepsilon'(n, p)>0$ such that if 
\begin{enumerate}

\item $\gamma\geq 0$ is a continuous function on $X$ such that 
$$|\{\gamma=0\}|_{\omega_X} \leq \varepsilon, ~ \{\gamma >0 \}~ \text{is connected},$$

\item $\omega\in \mathcal{W}(X, \omega_X, n, A, p, K; \gamma)$,

\item $s\in (1, \frac{2+2\varepsilon'}{2+\varepsilon'})$, 

\end{enumerate} 
then there exists $C=C(A, p, K, \varepsilon, \gamma, s)>0$ such that for any $u\in C^\infty(X)$ satisfying $\int_X u \omega^n=0$, 
$$  \| u  \|_{L^\infty(X)} \leq C \left( \frac{1}{[\omega]^n} \int_X |\nabla u|^{\frac{s}{s-1}} \omega^n \right)^{\frac{s-1}{s}}. $$

\end{lemma}

\begin{proof} By the Green's formula and integration by parts, we have
\begin{eqnarray*}
\left| u(x) \right| &=& \left|\int_{X \setminus \{x\}} \langle \nabla G(x, \cdot), \nabla u(\cdot)\rangle \omega^n \right|\\
&\leq& \left( \int_X |\nabla G(x, \cdot)|^s \omega^n\right)^{\frac{1}{s}} \left( \int_X |\nabla u|^{\frac{s}{s-1}} \omega^n \right)^{\frac{s-1}{s}}\\
&\leq& C( [\omega]^n )^{-\frac{s-1}{s}}\left( \int_X |\nabla u|^{s/(s-1)} \omega^n \right)^{\frac{s-1}{s}} \\
&=& C \left( \frac{1}{[\omega]^n}\int_X |\nabla u|^{s/(s-1)} \omega^n \right)^{\frac{s-1}{s}},
\end{eqnarray*}
for some uniform constant $C=C(A, p, K, \varepsilon, \gamma, s)>0$, after applying Proposition \ref{gres!}. 
\end{proof}
We remark that Proposition \ref{diaest} can also be proved by directly applying Lemma \ref{sobol}.


\section{Finite time solutions of the K\"ahler-Ricci flow} \label{seckrff}

\setcounter{equation}{0}

We will prove Theorem \ref{thm:main2} in this section by applying Theorem \ref{thm:main1}. The key is to  bound the $p$-Nash entropy from above and the volume form from below along the K\"ahler-Ricci flow. 

We consider the unnormalized K\"ahler-Ricci flow (\ref{krflow}) on a K\"ahler manifold $X$ with an initial K\"ahler metric $g_0$. Suppose the flow develops finite time singularity. Without loss of generality by rescaling, we can assume the singular time is given by 
$$T=\sup\{ t>0~|~[\omega_0] + t [K_X] >0\} = 1. $$
By choosing a smooth closed $(1,1)$-form $\chi \in K_X$, the K\"ahler-Ricci flow (\ref{krflow}) is equivalent to the following parabolic complex Monge-Amp\`ere flow. 
 \begin{equation}\label{maflow}
\left\{
\begin{array}{l}
{ \displaystyle \ddt{\varphi} = \log \frac{\left( \omega_0 + t  \chi+ \ddbar\varphi \right)^n}{\Omega},}\\
\\
\varphi|_{t=0} =0,
\end{array} \right.
\end{equation}
where $\Omega$ is a smooth volume form on $X$ satisfying 
$$\ddbar \log \Omega = \chi  \in [K_X].$$ We let $\omega_t= \omega_0+ t \chi $ and $\omega=\omega(t)= \omega_t + \ddbar \varphi.$

\begin{lemma} There exists $C>0$ such that $$\varphi\leq C, ~\ddt{\varphi}  \leq C$$ on $X\times [0, 1)$.

\end{lemma}

\begin{proof}  The upper bound for $\varphi$ follows directly from the maximum principle. Let $$u = t \ddt{\varphi} -\varphi-nt.$$ Then $u$ satisfies 
$$\left(\ddt{}- \Delta \right) u = - tr_{\omega}(\omega_0) \leq 0.$$
By the maximum principle, 
$$\sup_{X\times [0, 1)} u \leq \sup_X u(\cdot, 0) =0$$
 and so $t\ddt{\varphi}$ is also uniformly bounded from above. The lemma immediately follows by considering $t\in [1/2, 1)$ since $\ddt{\varphi}$ is uniformly bounded for $t\in [0, 1/2]$.
\end{proof}

We can now view the Monge-Amp\`ere flow as a family of complex Monge-Amp\`ere equations
$$(\omega_t + \ddbar\varphi)^n = e^{\ddt{\varphi}} \Omega$$
for $t\in [0, 1)$. 
If $[\omega_0]+[K_X]$ is big, 
$$\lim_{t\rightarrow 1} [\omega_t]^n = \left( [\omega_0]+ [K_X] \right)^n>0.$$

\begin{lemma} \label{8psi} There exists $\psi\in {\rm PSH}(X, \chi)$ such that $\omega_0+\chi+\ddbar \psi$ is a K\"ahler current on $X$. Furthermore, $\psi$ has analytic singularities and is smooth outside the locus of its singularities.

\end{lemma}

\begin{proof} The lemma is a consequence of \cite{DPn} (The regularization theorem 3.2).  
\end{proof}

We can assume that $\omega_0+ \chi + \ddbar \psi > \epsilon \omega_0$ for some $\epsilon>0$. 

\begin{lemma} \label{8c0} There exists $C>0$ such that  on $X\times [0, 1)$, we have 
$$\varphi \geq \psi - C .$$

\end{lemma}

\begin{proof} Let $u= \varphi-\psi$. Then $u$ is bounded below and tends to $\infty$ near the singular locus of $\psi$ for each $t\in [0, 1)$. $u$ satisfies the evolution equation
$$\ddt{u} = \log \frac{ \left( \omega_0+ \chi + \ddbar\psi - (1-t)\chi+ \ddbar u\right)^n}{\Omega}. $$
Let %
$$t' = \inf\{ 0<t<1~|~ \epsilon\omega_0 > 2(1-s)\chi, ~\rm for~ all~s\in (t, 1)  \}.$$
Obviously, $t'<1$. 
Suppose $\inf_{X\times[\max(1/2, t'), t_0)} u = u(z_0, t_0)$. Then $\psi$ is smooth at $z_0$. By applying the maximum principle, we have at $(z_0, t_0)$
\begin{eqnarray*}
\ddt{u} &\geq&  \log \frac{ \left( \omega_0+ \chi + \ddbar\psi - (1-t_0)\chi\right)^n}{\Omega}\\
&\geq& \log \frac{ \left( \epsilon \omega_0 - (1-t_0)\chi\right)^n}{\Omega}\\
&\geq& \log \frac{\omega_0^n}{\Omega} - C,
\end{eqnarray*}
for some uniform $C>0$. Therefore $u$ is uniformly bounded below for $t\in [0, 1)$.  principle. The lemma then immediately follows. 
\end{proof}

\begin{lemma} There exist $A, C>0$ such that 
$$\ddt{\varphi} \geq A\psi  - C. $$

\end{lemma}

\begin{proof} Let $u= \ddt{\varphi} + 2A(\varphi - \psi)$ for some fixed $A> 2\epsilon^{-1}>0$. Then the evolution for $u$ is given by
\begin{eqnarray*}
\left( \ddt{} - \Delta\right) u &=& 2A\ddt{\varphi} +tr_{\omega}(2A\omega_t +2A\ddbar\psi +\chi) -2nA\\
&=& 2A\ddt{\varphi} + tr_{\omega}(2A(\omega_1+ \ddbar\psi) + (1-2A(1-t))\chi ) - 2nA \\
&\geq& 2A\ddt{\varphi} + tr_{\omega} \left(2A\epsilon \omega_0 + (1-2A(1-t))\chi \right) -2nA \\
&\geq& 2A\ddt{\varphi} +2A\epsilon tr_{\omega} ( \omega_0) -2nA \\
&\geq& 2A\ddt{\varphi} +A\epsilon  \left( \frac{\omega_0^n}{\omega^n}\right)^{1/n}-2nA \\
&\geq& 2A\ddt{\varphi} + e^{- n^{-1} \ddt{\varphi} } -2nA.
\end{eqnarray*}
by choosing $A>>\epsilon^{-1}$ for  $t> 1-A^{-1}$.  
Let $p$ be the minimum point  of $u$ at $t> 1- A^{-1}$. Then $p$ does not lie in the singular locus of $\psi$ and by applying the maximum principle,we have  
$$\ddt{\varphi}(p) \geq - C$$ 
for some uniform constant $C>0$. Hence $u(p)$ is uniformly bounded below by applying Lemma \ref{8c0}. The lemma then immediately follows. 
\end{proof}

\begin{corollary}\label{8wass} For any $p>n$, there exists $C>0$ such that for all $t \in[0, 1)$
$$\mathcal{N}_{X, \omega_0, p}(\omega(t)) \leq C.$$
Furthermore, for any $p>n$, there exist $A, B, K>0$  such that for all $t\geq 0$, 
$$\omega(t)\in \mathcal{W}\left(X, \omega_0, n, A, p, K; e^{B\psi -B}\right), $$
where $\psi$ is defined in Lemma \ref{8psi}. 

\end{corollary} 

\begin{proof} By combining the previous lemmas, there exist $C_1, C_2>0$ such that on $X\times [0, 1)$, we have
$$C_2^{-1} e^{C_1\psi} (\omega_0)^n \leq \omega^n \leq C_2 (\omega_0)^n.$$
The corollary immediately follows. 

\end{proof}

\noindent{\it Proof of Theorem \ref{thm:main2}.}  The assumption of Theorem \ref{thm:main1} is satisfied due to Corollary \ref{8wass}. Theorem \ref{thm:main2} immediately follows. \qed


\section{Long time solutions of the K\"ahler-Ricci flow } \label{seckrfl}

\setcounter{equation}{0}

We will prove Theorem \ref{thm:main3} in this section as an application of Theorem \ref{thm:main1}. As in the previous section, we will  bound the $p$-Nash entropy from above and the volume form from below along the K\"ahler-Ricci flow.

 Let $X$ be a K\"ahler manifold with nef $K_X$ and nonnegative Kodaira dimension. For any smooth closed $(1,1)$-form $\chi\in K_X$, we can find a smooth volume form $\Omega$ such that
$$\chi = \ddbar \log \Omega, ~\int_X \Omega = 1. $$
We let $\omega_t =(1-e^{-t}) \chi + e^{-t} \omega_0$. The numerical dimension of $K_X$ is defined by
$$\kappa=\kappa(X)= \max\{ k\geq 0: [K_X]^k\neq 0~{\rm in}~ H^{k,k}(X, \mathbb{R}) \}.$$
We will assume ${\rm{Kod}}(X)\geq 0$, the numerical dimension $\kappa(X) \geq {\rm{Kod}}(X)\geq 0$. The normalized K\"ahler-Ricci flow (\ref{nkrflow}) is equivalent to the following parabolic complex Monge-Amp\`ere equation
 \begin{equation}\label{nmaflow}
\left\{
\begin{array}{l}
{ \displaystyle \ddt{\varphi} = \log \frac{e^{(n-\kappa) t} \left(\omega_t+ \ddbar\varphi \right)^n}{\Omega} -\varphi,}\\
\\
\varphi|_{t=0} =0.
\end{array} \right.
\end{equation}
 We let $\omega=\omega(\cdot, t)= \omega_t + \ddbar\varphi$  solving the Monge-Amp\`ere flow (\ref{nmaflow}).

 We first derive the volume growth for $(X, g(t))$.
\begin{lemma} \label{algvol} There exists $C>0$ such that for all $t\geq0$, we have
$$C^{-1}  e^{-(n-\kappa)t} \leq [\omega_t]^n \leq C e^{-(n-\kappa)t}.$$ 

\end{lemma}

\begin{proof} Let $\alpha = [\omega_0]-K_X$. Then by the definition of $\kappa$, we have
$$(K_X)^\kappa \cdot \alpha^{n-\kappa} >0$$
and
\begin{eqnarray*}
e^{(n-\kappa)t} [\omega_t]^n &=& \sum_{l=0}^n C_n^l e^{(l-\kappa )t} (K_X)^l \cdot \alpha^{n-l} \\
&=&\sum_{l=0}^\kappa C_n^l e^{(l-\kappa )t} (K_X)^l \cdot \alpha^{n-l} \\
&=&C_n^\kappa (K_X)^\kappa \cdot \alpha^{n-\kappa} + O(e^{-t}).
\end{eqnarray*}
This proves the lemma. 
\end{proof}

\begin{lemma} \label{ltbas}There exists $C>0$ such that for all $t\in [0, \infty)$, 
$$-C \leq \sup_{X} \varphi(\cdot, t) \leq C, ~ \sup_X \left( \ddt{\varphi}+\varphi\right)(\cdot, t)  \leq C. $$

\end{lemma}

\begin{proof}  By Jensen inequality and Lemma \ref{algvol}, 

\begin{eqnarray*}
\ddt{}  \left( \int_X \varphi(\cdot, t)  \Omega \right) &=& \int_X \left(\log \frac{ e^{(n-\kappa)t}\omega^n}{\Omega}\right) \Omega - \int_X \varphi  (\cdot, t) \Omega\\
&\leq&   \log \left(  \int_X e^{(n-\kappa)t}\omega^n  \right)  -\int_X \varphi  (\cdot, t)\Omega\\
&\leq& \log \left( e^{(n-\kappa)t} [\omega_t]^n\right) -\int_X \varphi  (\cdot, t)\Omega\\
&\leq& C - \int_X \varphi (\cdot, t) \Omega.
\end{eqnarray*}
Hence $\int_X \varphi (\cdot, t) \Omega$ is uniformly bounded above. Since $\varphi \in {\rm PSH}(X, \omega_t) \subset {\rm PSH}(X, A \omega_0)$ for some fixed sufficiently large $A>0$, by the mean value theorem for plurisubharmonic functions, there exists $C>0$ such that for any $t\in [0, \infty)$ and $x\in X$, 
$$\varphi (x, t) \leq \int_X\varphi(\cdot, t) \Omega + C. $$
This proves the uniform upper bound for $\varphi$.

Now we let $u= \ddt{\varphi} - e^{-t} \varphi$. Then the evolution for $u$ is given by
\begin{eqnarray*}
\ddt{u}  &=& \Delta u - u  - e^{-t} \ddt{\varphi} - e^{-t} tr_\omega( \omega_t +\omega_0- \chi )+ n-\kappa + ne^{-t} \\
&=&\Delta u - u  + e^{-t} \log\frac{\Omega}{\omega^n}  - e^{-t} tr_\omega( \omega_0+e^{-t}(\omega_0 - \chi )) \\
&& + e^{-t} \ddt{\varphi} + n-\kappa + e^{-t}(n- (n-\kappa)t).
\end{eqnarray*}
Then there exist $C_1, C_2>0$ such that for all $t\geq 0$, 
\begin{eqnarray*}
\ddt{u}  & \leq & \Delta u - u  - \frac{1}{2} e^{-t} \left( tr_\omega( \omega_0 ) - \log \frac{\omega_0^n}{\omega^n} \right)  + C_1\\
&\leq & \Delta u - u- C_2.
\end{eqnarray*}
 The uniform upper bound for $u$ immediately follows from the maximum principle. Therefore. 
 $$\ddt{\varphi} + \varphi = u + (1+ e^{-t}) \varphi$$ is uniformly bounded above.
 
We now prove the lower bound for $\sup_X \varphi(\cdot, t)$.  By the upper bound of $\ddt{\varphi} - e^{-t} \varphi$ and by Lemma \ref{algvol}, there exist $c_1, c_2>0$ such that for all $t\geq 0$, we have
\begin{eqnarray*}
e^{(1+e^{-t})\sup_X \varphi(\cdot, t) } & \geq & c_1e^{ (1+e^{-t}) \sup_X \varphi  + \sup_X \left(\ddt{\varphi}  - e^{-t} \varphi \right) }  \\
&\geq & \int_X e^{\ddt{\varphi} + \varphi} \Omega\\
&=& \int_X e^{(n-\kappa)t}  \omega^n \\
&\geq& c_2.
\end{eqnarray*}
This completes the proof of the theorem.  \end{proof}

Let 
\begin{equation}
\mathcal{V}_t(x)= \sup\{ \phi(x) ~|~\omega_t + \ddbar \phi \geq 0, ~\phi\leq 0 \}
 \end{equation}
be the extremal function associated to $\omega_t$ for any $t\in [0, \infty)$. We let $V_\infty$ be the extremal function associated to $\chi$.  
Since we assume $\kappa>0$, there exists a holomorphic section $\sigma \in |mK_X|$ for some sufficiently large $m$. Let $h$ be the hermitian metric on $mK_X$ with ${\rm Ric}(h)=m\chi$ and $\sup_X |\sigma|_h^2 = 1. $ Then 
\begin{equation}\label{ltvinf}
\mathcal{V}_\infty \geq \frac{1}{m} \log |\sigma|^2_h
\end{equation}
because $$\chi+ \frac{1}{m}\ddbar \log |\sigma|^2_h = [\sigma] \geq 0.$$

The following lemma is obvious by definition. 

\begin{lemma} For any $t_1\leq t_2$, 
$$\mathcal{V}_\infty \leq \mathcal{V}_{t_2} \leq \mathcal{V}_{t_1}. $$

\end{lemma}

\begin{lemma} \label{ltc0} There exists $C>0$ such that on $X\times [0, \infty)$, we have
\begin{equation}
\varphi(\cdot, t)   \geq \mathcal{V}_t - C\geq  \mathcal{V}_\infty - C. 
\end{equation}

\end{lemma}

\begin{proof} By Lemma \ref{algvol} and Lemma \ref{ltbas}, There exists $C>0$ such that on $X\times [0, \infty)$, the normalized volume measure is uniformly bounded above by a smooth volume form
$$\frac{\omega^n}{[\omega]^n} \leq C \Omega. $$
The lemma then follows from Proposition \ref{linf} and the uniform lower bound of $\sup_X \varphi(\cdot, t)$.
\end{proof}

\begin{lemma} \label{t2est} There exists $C>0$ such that on $X\times [0, \infty)$, we have
$$\frac{\partial^2 \varphi }{\partial t^2}  + \ddt{\varphi} \leq C. $$
\end{lemma}
\begin{proof}

Let $u = \ddt{\varphi} + \varphi$. Let ${\rm R}$ be the scalar curvature of $\omega$. Since the scalar curvature of the K\"ahler-Ricci flow is uniformly bounded below, there exists $C>0$ such that on $X\times [0, \infty)$, we have
$$\Delta u = \Delta \left( \log \frac{\omega^n}{\Omega} \right) = - {\rm R} -  tr_\omega(\chi) \leq C - tr_\omega(\chi).$$
On the other hand, 
\begin{eqnarray*}
\ddt{u}  &=&   \frac{\partial^2 \varphi }{\partial t^2}  + \ddt{\varphi} \\
&\leq& \Delta u + tr_\omega(\chi) - \kappa\\
&\leq &C -\kappa. 
\end{eqnarray*}
This proves the lemma.
\end{proof}

 \begin{lemma} \label{ode} Suppose $f=f(x)$ is a smooth function for $x\geq 0$ and it satisfies the differential inequality 
 $$f'' + f' \leq 1, ~|f| \leq A$$ 
 for all $x\geq 0$ and
 %
 %
 for some fixed $A> 1$. Then for any $x\geq 0$, we have, $$f' \geq - 5A. $$
 
 \end{lemma}

\begin{proof}  Suppose the lemma fails. Then there exists $x_0\geq 0$, such that 
$$f'(x_0) < - 5 A. $$ Since $|f|\leq A$, there must exist an $a\in (0,1)$ such that
$$f'(x_0+a) = - 2 A. $$ Otherwise, $f(x) <  f(x_0) - 2A \leq - A$ for $x\in (x_0, x_0+1)$, contradicting the assumption.

By our assumption, $(f'+f - x)' \leq 0$,  and so 
$$          f'(x_0+a) + f(x_0+a) - (x_0+a) \leq f'(x_0) + f(x_0) - x_0.$$
This implies that
$$-2A =f'(x_0+a) \leq f'(x_0) + f(x_0)-f(x_0+a) +a \leq -5A + 2A +1= -3A +1.$$
This leads to contradiction as $A > 1$.  \end{proof}

\begin{lemma} \label{ltvob}There exists $C>0$ such that on $X\times [0, \infty)$, we have 
$$\ddt{\varphi} \geq  5 \mathcal{V}_\infty - C. $$

\end{lemma}

\begin{proof}  By Lemma \ref{ode} and Lemma \ref{ltc0}, there exists $C>0$ such that for any $x\in X$ and $t\geq 0$,we have 
$$\frac{\partial^2 \varphi}{\partial t^2} + \ddt{\varphi} \leq C, ~ - \mathcal{V}_\infty  - C \leq \varphi \leq C$$
The lemma is then proved by directly applying Lemma \ref{ode}. 
\end{proof}

We then immediately have the following corollary by combining Lemma \ref{ltbas}, Lemma \ref{ltvob} and (\ref{ltvinf}).

\begin{corollary} \label{volkfr} There exists $C=C(X, g_0)>0$ such that on $X\times [0, \infty)$, we have
$$ C^{-1} \exp \left(\frac{6}{m} \log |\sigma|^2_h  \right)      \leq    \frac{ 1}{[\omega]^n } \frac{\omega^n}{\Omega}  \leq C .$$

\end{corollary}

\begin{proof} By Lemma \ref{algvol}, there exists $C>0$ such that $$C^{-1} {\rm exp} \left( \ddt{\varphi} + \varphi \right) \leq \frac{1}{[\omega]^n} \frac{\omega^n}{\Omega} \leq C {\rm exp} \left( \ddt{\varphi} + \varphi \right). $$ The corollary is then a direct consequence of Lemma \ref{ltvob}, Lemma \ref{ltc0} and (\ref{ltvinf}). 
\end{proof}

Corollary \ref{volkfr} implies the bound for $p$-Nash entropy and a pointwise lower bound  for $\omega(t)$.

\begin{corollary}\label{9wass} For any $p>n$, there exists $C>0$ such that for all $t \geq 0 $
$$\mathcal{N}_{X, \omega_0, p}(\omega(t)) \leq C.$$
Furthermore, for any $p>n$, there exist $A, B, K>0$  such that for all $t\geq 0$, 
$$\omega(t)\in \mathcal{W}(X, \omega_0, n, A, p, K; B^{-1} |\sigma|_h^{2B}), $$
where $\psi$ is defined in Lemma \ref{8psi}. 

\end{corollary}

\noindent{\it Proof of Theorem \ref{thm:main3}.}  The assumption of Theorem \ref{thm:main1} is satisfied due to Corollary \ref{9wass}. Theorem \ref{thm:main3} immediately follows. \qed

%


\section{Family of projective manifolds} \label{secfam}

\setcounter{equation}{0}

In this section, we will extend Theorem \ref{thm:main1} to a projective family with not only varying K\"ahler classes but also complex structures. Such extensions will allow us to obtain fibre diameter estimates for degeneration of canonical K\"ahler metrics on special fibrations. 

Let 
\begin{equation}\label{famdef}
\pi: \mathcal{X} \subset \mathbb{CP}^N \times \mathbb{D}  \rightarrow \mathbb{D}
\end{equation}
 be a projective  family over a unit disk $\mathbb{D}$ with $\mathcal{X}_t=\pi^{-1}(t)$ being a smooth $n$-dimensional projective manifold for each $t\in \mathbb{D}^*$.  We further assume that $\pi$ is proper and flat,  and the central fibre $\mathcal{X}_0=\pi^{-1}(0)$ is reduced and irreducible. 
Let $\theta_t$ be the restriction of the Fubini-Study metric $\theta$ of $\mathbb{CP}^N$ to $\mathcal{X}_t$.

\begin{theorem} \label{famdia} Let $\pi: \mathcal{X}\subset \mathbb{CP}^N\times \mathbb{D} \rightarrow \mathbb{D}$ be a projective family defined as above in (\ref{famdef}). Let $\gamma$ be a nonnegative continuous function on $\mathbb{CP}^N\times \mathbb{D}$ such that  $\{ \gamma=0\}$ is a  proper subvariety of $\mathbb{CP}^N\times \mathbb{D}$ and $\{ \gamma=0\}$ does not contain $\mathcal{X}_t$ for any $t\in \mathbb{D}$.  Then for any $A>0$, $p>n$ and $K>0$,  there exists $C=C(A, p, K, \gamma)>0$ such that for any $t\in \frac{1}{2} \mathbb{D}^*$ and any K\"ahler metric $\omega$ on $\mathcal{X}_t$, if 
\begin{equation}
\mathcal{N}_{\mathcal{X}_t, \theta_t, p}(\omega) \leq  K, ~ ~ [\omega] \leq  A [\theta_t], ~ ~
\frac{1}{{\rm Vol}_\omega(\mathcal{X}_t)} \frac{\omega^n}{ \theta_t^n}  \geq \gamma, 
\end{equation}
then

$${\rm diam}({\mathcal{X}_t}, \omega) \leq C, $$
and
$$ \int_{\mathcal{X}_t} \big( |G(x, \cdot)| + |\nabla  G(x, \cdot)|\big) \omega^n + \left( - \inf_{y\in {\mathcal{X}_t}} G(x, y) \right)  {\rm Vol}_\omega(\mathcal{X}_t)\leq  C,  $$
for any $x\in \mathcal{X}_t$, where $G$ is the Green's function of $(\mathcal{X}_t, \omega)$.
Moreover, there exist a uniform constant $c=c(A, p, K, \gamma)>0$ and $\alpha = \alpha(n, p)>0$ such that for any $x\in \mathcal{X}_t$ and $R\in (0,1]$,
$$
\frac{{\mathrm{Vol}}_{\omega} (B_{\omega}(x, R))}{{\rm Vol}_{\omega}(\mathcal{X}_t)}\ge c R^\alpha.
$$

\end{theorem}

The proof of Theorem \ref{famdia} is almost identical to the proof of Theorem \ref{thm:main1} except for a few technical differences as presented below.

\begin{lemma}\label{famcutoff}  Let $\pi: \mathcal{X}\subset \mathbb{CP}^N\times \mathbb{D} \rightarrow \mathbb{D}$ be a projective family defined as above in (\ref{famdef})  and let $S$ be a subvariety of $\mathbb{CP}^N\times \mathbb{D}$. If  $S$ does not contain $\mathcal{X}_t$ for any $t\in \mathbb{D}$ and  $\textnormal{Sing}(\mathcal{X}_0)\times\{0\}  \subset S$, then for any $\epsilon>0$ and $\mathcal{K} \subset\subset \mathcal{X}\setminus S$, there exists $\rho_\epsilon \in C^\infty( \left( \mathbb{CP}^N \times \mathbb{D} \right)  \setminus S)$ such that 
\begin{enumerate}
\item $0\leq \rho_\epsilon\leq 1$,

\item $\textnormal{Supp} \rho_\epsilon \subset\subset \left( \mathbb{CP}^N \times \mathbb{D} \right) \setminus S$, 

\item $\rho_\epsilon =1$ on $\mathcal{K}$, 

\item $\int_{\mathcal{X}_t} |\nabla \rho_\epsilon|^2 \theta_t^n < \epsilon$, for any $t\in \frac{1}{2}\mathbb{D}$.

\end{enumerate}

\end{lemma}

\begin{proof}

We can assume $S$ is a union of smooth divisors after replacing $\mathbb{CP}^N \times \mathbb{D}$ by its blow-ups $\mathcal{Y}$.  Let $\sigma$ be the defining section for $S$ and $h$ be a smooth hermitian metric on the line bundle associated to $S$. Without loss of generality, we can assume that $|\sigma |_h^2 \leq 1$. We can always pick a K\"ahler metric $\tau$ on $\mathcal{Y}$ such that $\tau > \ric(h) = -i\partial \bar \partial \log h$ and $\tau_{\mathbb{CP}^N}\geq \theta$ after slightly shrinking $\mathbb{D}$. For simplicity, we identify $\theta$ with its pullback from $\mathbb{CP}^N \times \mathbb{D}$. 

Let $F$ be the standard smooth cut-off function on $[0, \infty)$ with $F=1$ on $[0, 1/2]$ and $F=0$ on $[1, \infty)$. 
We then let 
$$\eta_\epsilon = \max ( \log |\sigma|^2_h, \log \epsilon) . $$
For $\epsilon<1$, we have $\log \epsilon \leq \eta_\epsilon \leq 0$. 
Then obviously, $\eta_\epsilon \in {\rm PSH}(\mathcal{Y}, \tau) \cap C^0(\mathcal{Y})$.
Now we let 
$$\rho_\epsilon = F\left(\frac{\eta_\epsilon}{\log \epsilon} \right).$$
Then  $\rho_\epsilon =1 $ on $\mathcal{K}$ if $\epsilon$ is sufficiently small. Let $\mathcal{X}'_t$ be the proper transform of $\mathcal{X}_t$ by the blow-up and $\rho_{\epsilon, t}= \rho_{\epsilon}|_{\mathcal{X}_t}$, $\eta_{\epsilon, t}= \eta_{\epsilon}|_{\mathcal{X}_t}$, $\tau_t= \tau|_{\mathcal{X}_t}$. For simplicity, we identify $\rho_\epsilon$ and $\theta_t$ with their pullbacks from $\mathcal{X}_t$ to $\mathcal{X}_t'$. Straightforward calculations give
\begin{eqnarray*}
&&\int_{\mathcal{X}'_t} \sqrt{-1} \partial \rho_{\epsilon, t} \wedge \dbar \rho_{\epsilon, t} \wedge \theta_t^{n-1}  \\
&=&(\log \epsilon)^{-2}  \int_{\mathcal{X}'_t}  (F')^2 \sqrt{-1}\partial \eta_{\epsilon, t} \wedge \dbar \eta_{\epsilon, t} \wedge \theta_t^{n-1}\\
&\leq& C (\log \epsilon)^{-2} \int_{\mathcal{X}'_t}  (-\eta_{\epsilon, t}) \ddbar \eta_{\epsilon, t} \wedge \theta_t^{n-1} \\
&=& C(\log \epsilon)^{-2} \int_{\mathcal{X}'_t} (-\eta_{\epsilon, t})(\tau_t+ \ddbar \eta_{\epsilon, t}) \wedge \theta_t^{n-1}
+ C(\log \epsilon)^{-2} \int_{\mathcal{X}'_t}  \eta_{\epsilon, t} ~ \tau_t \wedge \theta_t^{n-1}\\
&\leq& C(-\log \epsilon)^{-1} \int_{\mathcal{X}'_t}  (\tau_t+ \ddbar \eta_{\epsilon, t}) \wedge \theta_t^{n-1} \\
&\leq& C(-\log \epsilon)^{-1} [\tau]^n\cdot \mathcal{X}'_t \\
&\leq& C(-\log \epsilon)^{-1}  \rightarrow 0
\end{eqnarray*}
as $\epsilon \rightarrow 0$, where the constant $C$ is independent of $t\in \frac{1}{2} \mathbb{D}$.
Therefore we obtain  $\rho_\epsilon \in C^0({\mathcal{X}'_t} )$ satisfying the conditions in the lemma. The lemma is then proved by smoothing $\rho_\epsilon$ on $\left( {\rm Supp}~ \rho_\epsilon \right) \setminus \mathcal{K}$.  
\end{proof}

\begin{lemma} \label{alphinv}

Let $\pi: \mathcal{X}\subset \mathbb{CP}^N\times \mathbb{D} \rightarrow \mathbb{D}$ be a projective family defined as above.   Then there exist $\alpha>0$ and $C>0$ such that for any $t\in \frac{1}{2} \mathbb{D}^*$ and $\varphi \in {\rm PSH}(\mathcal{X}_t, \theta_t)$, we have
\begin{equation}
\int_{\mathcal{X}_t} e^{- \alpha \left(\varphi - \sup_{\mathcal{X}_t} \varphi \right)} \theta_t^n \leq C. 
\end{equation}

\end{lemma}

\begin{proof} The lemma is an immediate consequence of the results in \cite{dGG} (c.f. Theorem 3.4 in \cite{dGG}).
\end{proof}

\begin{lemma} \label{faminfty}

Let $\pi: \mathcal{X}\subset \mathbb{CP}^N\times \mathbb{D} \rightarrow \mathbb{D}$ be a projective family defined as above.  For any $A>0$, $p>n$, $K>0$, there exists $C=C( A, p, K)>0$ such that for any $t\in \frac{1}{2} \mathbb{D}^*$,  if $\eta$ is a smooth closed $(1,1)$-form on $\mathcal{X}_t$ with $\eta\leq A \theta_t$ and if $\omega=\eta + \ddbar \varphi $ is a K\"ahler form on $\mathcal{X}_t$ satisfying
\begin{equation}
\mathcal{N}_{\mathcal{X}_t, \theta_t,p}(\omega)  \leq  K,
\end{equation}
then
$$ \|\varphi-\sup_X \varphi-  \mathcal{V}_\eta  \|_{L^\infty(\mathcal{X}_t)} \leq C, $$ 
where  $\mathcal{V}_\eta= \sup\{ u\in {\rm PSH}(\mathcal{X}_t, \eta): ~u\leq 0\}$ is the envelope of non-positive $\eta$-PSH functions.

\end{lemma}

\begin{proof} By Lemma \ref{alphinv}, the $\alpha$-invariant for ${\rm PSH}(\mathcal{X}_t, A\theta_t)$ is uniformly bounded for all $t\in \frac{1}{2}\mathbb{D}^*$. Since $[\omega]\leq A [\theta_t]$, we can choose a smooth closed $(1,1)$-form $\eta\in [\omega]$ with $\eta \leq A \theta_t$ and  then ${\rm PSH}(\mathcal{X}_t, \eta) \subset {\rm PSH}(\mathcal{X}_t, A \theta_t)$. We can consider the following Monge-Amp\`ere equation
$$(\eta+\ddbar \varphi)^n = \omega^n, ~ \sup_X \varphi=0.$$
The right hand side has uniformly bounded $p$-Nash entropy with respect to $\theta_t$ for some $p>n$. The lemma follows the work of \cite{EGZ, DP, BEGZ, FGS, GPTW} as generalizations of Kolodziej's work \cite{K} (c.f. Proposition \ref{linf}). 
\end{proof}

\begin{lemma} \label{lemma key 2}

Let $\pi: \mathcal{X}\subset \mathbb{CP}^N\times \mathbb{D} \rightarrow \mathbb{D}$ be a projective family defined as above.  Then for any $A>0$, $p>n$, $K>0$, there is a uniform constant $C=C( A, p, K)>0$ such that for  any $t\in \frac{1}{2} \mathbb{D}^*$, any K\"ahler metric $\omega$ on $\mathcal{X}_t$ satisfying
\begin{equation}
\mathcal{N}_{\mathcal{X}_t, \theta_t,p}(\omega)  \leq  K, ~ ~ [\omega] \leq  A [\theta_t], 
\end{equation}
and 
$v\in C^2(\mathcal{X}_t)$ satisfying 
$$|\Delta_{\omega} v| \le 1\, \mbox{ and } \int_{\mathcal{X}_t} v \omega^n = 0,$$
we have
$$\sup_{\mathcal{X}_t}  |v|\le C\left(1 + \frac{1}{[\omega]^n} \int_{\mathcal{X}_t} |v| \omega^n \right).$$

\end{lemma}

\begin{proof} The lemma can be proved by the same argument for Lemma \ref{lemma key} and Corollary \ref{lemma 3} with the estimate established in Lemma \ref{faminfty}. 
\end{proof}

\begin{lemma} \label{lemma6v2} Let $\pi: \mathcal{X}\subset \mathbb{CP}^N\times \mathbb{D} \rightarrow \mathbb{D}$ be a projective family defined as above. Let $\gamma$ be a nonnegative smooth function on $\mathbb{CP}^N \times \mathbb{D}$ such that  $\{ \gamma=0\}$ is a  subvariety of $\mathbb{CP}^N \times \mathbb{D}$ and $\{ \gamma=0\}$ does not contain $\mathcal{X}_t$ for any $t\in \mathbb{D}$.  Then for any $A>0$, $p>n$, $K>0$, there exists $C=C( A, p, K,\gamma)>0$ such that for any $t\in \frac{1}{2} \mathbb{D}^*$, any K\"ahler metric $\omega$ on $\mathcal{X}_t$ satisfying
\begin{equation}
\mathcal{N}_{\mathcal{X}_t, \theta_t,p}(\omega)  \leq  K, ~ ~ [\omega] \leq  A [\theta_t], ~ ~
\frac{\omega^n}{ \theta_t^n}  \geq \gamma, 
\end{equation}
and  $v\in C^2(\mathcal{X}_t)$ satisfying
$$|\Delta_{\omega} v| \le 1\, \mbox{ and } \int_{\mathcal{X}_t} v \omega^n = 0,$$
we have 
$$\frac{1}{[\omega]^n} \int_{\mathcal{X}_t} |v| \omega^n\le C.$$

\end{lemma}

\begin{proof}. We will follow the same proof of Lemma \ref{lemma 6} by the argument of contradiction. Suppose Lemma \ref{lemma6v2} fails. Then there exist a sequence of $t_j \in \frac{1}{2}\mathbb{D}$, K\"ahler metric $\omega_j$ of $\mathcal{X}_{t_j}$ and $v_j\in C^2(\mathcal{X}_{t_j})$ satisfying 
$$\mathcal{N}_{\mathcal{X}_{t_j}, \theta_{t_j},p}(\omega_j)  \leq  K, ~ ~ [\omega_j] \leq  A [\theta_{t_j}], ~ ~
\frac{\omega_j^n}{ \theta_{t_j}^n}  \geq \gamma, $$
and 
$$\Delta_{\omega_j} v_j = h_j,\quad \int_X v_j \omega_j^n = 0,$$
for some $|h_j|\le 1$, and as $j\to\infty$
\bea
\frac{1}{V_j} \int_X |v_j| \omega_j^n = : N_j \to \infty,
\eea
where $V_j=[\omega_j]^n = \int_{\mathcal{X}_{t_j}} \omega_j^n$.
We consider $\tilde v_j$ defined by $\tilde v_j = v_j/N_j$, which clearly satisfies
$$|\Delta_{\omega_j} \tilde v_j| = |h_j|/ N_j\to 0,\, \mbox{ and }\frac{1}{V_j} \int_X |\tilde v_j| \omega_j^n = 1.$$

We define $F_j$ on $\mathcal{X}_{t_j}$ by 
$$\omega_j^n = [\omega_j]^n e^{F_j} \theta_{t_j}^n. $$ 
Applying Lemma \ref{lemma key 2} to $\tilde v_j$, we get $\sup_X |\tilde v_j|\le C$ for some uniform $C>0$. From the equation of $\tilde v_j$ and integration by parts, we see that
$$ \frac{1}{V_j} \int_X |\nabla \tilde v_j|_{\omega_j}^2 e^{F_{j}} \theta_{t_j} ^n=\frac{1}{V_j}\int_{\mathcal{X}_{t_j}}  |\nabla \tilde v_j|_{\omega_{t_j}}^2 \omega_j^n  \to 0.   $$

Let $\cS=\{ \gamma =0\} \cup \left( {\rm Sing} \mathcal{X}_0 \times \{0\} \right)$. For each $k>0$, we pick a sequence of  a pair open subsets $U_k \subset\subset V_k \subset\subset \left( \mathcal{X}\times \mathbb{D} \right) \setminus \cS$  and cut-off functions $\rho_k$  in Lemma \ref{famcutoff} satisfying the following conditions. 

\begin{enumerate}

\item $V_k \subset V_{k+1}$, $U_k \subset U_{k+1}$,

\medskip

\item $\lim_{k \rightarrow \infty} U_k = \left( \mathcal{X}\times\mathbb{D} \right) \setminus \cS$ and ${\rm Vol}_{\theta_t}( \mathcal{X}_t \setminus U_k) < \frac{1}{k} $ for each $t\in \frac{1}{2}\mathbb{D}$, 

\medskip

\item $\rho_k = 1$ on $U_k$ and $\rho_k = 0$ on $\left( \mathcal{X} \times \mathbb{D} \right) \setminus V_k$, 

\medskip

\item $\int_{\mathcal{X}_t}  |\nabla \rho_k|_{\theta_t} ^2 \theta_t ^n < \frac{1}{k^2}, $ for each $t\in \mathbb{D}$.

\end{enumerate}

Let $${\mathcal{C}}_k = \sup_{t\in \frac{1}{2}\mathbb{D}} \sup_{\mathcal{X}_t\cap V_k} e^{-\gamma}.$$
Then for any $j$ and $t\in\frac{1}{2}\mathbb{D}$,  we have  $$ \inf_{V_k\cap \mathcal{X}_t } e^{F_j} \geq \left( {\mathcal{C}}_k \right)^{-1}.$$
After passing to a subsequence and relabelling $j$, we can assume that  
$$\int_{\mathcal{X}_{t_j}} |\nabla \tilde v_j|^2_{\omega_j } \omega_{j}^n \leq  j^{-2}  {\mathcal{C}}_j^{-1}. $$
We now define 
$$u_j = \rho_j \tilde v_j .$$
%
%
Then there exists $C>0$ such that for all $j$, we have 
\begin{eqnarray*}
&&\int_{\mathcal{X}_{t_j}} |\nabla u_j|_{\theta_{t_j}} \theta_{t_j} ^n \\
&\le& \int_{\mathcal{X}_{t_j}}  |\tilde v_j| |\nabla \rho_j|_{\theta_{t_j}} \theta_{t_j}^n + \int_X \rho_kj|\nabla \tilde v_j|_{\theta_{t_j}} \theta_{t_j}^n\\
&\le& Cj^{-1}+ \Big(\int_{\mathcal{X}_{t_j}}  |\nabla \tilde v_j|_{\omega_j }^2 e^{F_j} \theta_{t_j}^n \Big)^{1/2} \Big( \int_{\mathcal{X}_{t_j}}  \rho_j^2 e^{-F_j}  \tr_{\theta_{t_j}} (\omega_j)  \theta_{t_j} ^n \Big)^{1/2}\\
&\le& Cj^{-1} + j^{-1} \Big( \int_{\mathcal{X}_{t_j}}  \omega_j \wedge \theta_{t_j}^{n-1} \Big)^{1/2}\\
&\le& 2C j^{-1}.
\end{eqnarray*}
Without loss of generality, we can assume that $t_j \rightarrow 0$. For any $\mathcal{K}\subset \subset \left( \mathcal{X}\times \mathbb{D} \right) \setminus \cS$, we can pick a sufficiently small $0\in \Delta \in \mathbb{D}$ such that $\pi^{-1}(\Delta) \cap K \subset \cup_i \left( W_i \times \Delta \right)$, with holomorphic local coordinates $(z_1, ..., z_n, t)$ with $t\in \Delta$. Since $u_j$ is uniformly bounded in $W^{1,1}(\mathcal{X}_{t_j}, \theta_{t_j})$, by Sobolev embedding theorem, after passing to a subsequence, we can assume that $u_j$ converges to $u_\infty$ in $L^1_{\rm loc}(\mathcal{X}_0 \setminus \cS, \theta_0)$. Furthermore, $u_\infty$ and $u_j$ are uniformly bounded in $L^\infty(\mathcal{X}_0)$ and $L^\infty(\mathcal{X}_{t_j})$.

Since $\lim_{j \rightarrow \infty} \int_{\mathcal{X}_{t_j}} |\nabla u_j|_{\theta_{t_j}} \theta_{t_j}^n =0$,  for any test function $f\in C^\infty(\mathcal{X}_0\setminus \cS)$ with compact support in $\mathcal{X}_0\setminus \cS$, we can extend $f$ to $C^\infty \left( \left( \mathcal{X} \times \mathbb{D} \right) \setminus \cS \right)$ with compact support in $\left( \mathcal{X}\times \mathbb{D} \right) \setminus \cS$ and 
\begin{eqnarray*}
\left| \int_{\mathcal{X}_{t_j}} u_j (\Delta_{\theta_{t_j}}  f) \theta_{t_j}^n \right| &=&  \int_{\mathcal{X}_{t_j}}  | \nabla u_j|_{\theta_{t_j}} | \nabla f |_{\theta_{t_j}} \theta_{t_j}^n\\
&\leq& \left( \sup_{t\in \mathbb{D}, \mathcal{X}_{t_j}}|\nabla f|_{\theta_t} \right) \int_{\mathcal{X}_{t_j}} |\nabla u_j|_{\theta_{t_j}}\theta_{t_j}^n \\
&\rightarrow & 0
\end{eqnarray*}
as $j \rightarrow \infty$. Therefore 
$$\int_{\mathcal{X}_0\cap \left( W_i \times \Delta\right)} u_\infty \left( \Delta_{\theta_0} f  \right) \theta_0^n = 0$$
and by Weyl's lemma for the Laplace equation, $u_\infty$ solves the Laplace equation $\Delta u_\infty=0$ on $\mathcal{X}_0 \setminus \cS$   and $u_\infty \in C^\infty(\mathcal{X}_0\setminus \cS)$.  
This immediately implies that 
$$\int_{\mathcal{X}_0} \left( \rho_{j, 0} \right)^2 u_\infty \left( \Delta_{\theta_0} u_\infty \right) \theta_0^n =0$$
for any $j$, where $\rho_{j, 0}$ is the restriction of $\rho_j$ to $\mathcal{X}_0$.
  Then
$$\int_{\mathcal{X}_0}( \rho_{j, 0})^2 |\nabla u_\infty|^2 _{\theta_0} \theta_0^n \leq 8 \sup_{\mathcal{X}_0} |u_\infty|^2 \left( \int_{\mathcal{X}_0} |\nabla \rho_{j, 0}|^2_{\theta_0} \theta_0^n \right) \rightarrow 0$$
as $j\rightarrow \infty$ by the choice of $\rho_j$. Therefore $u_\infty$ is a constant on $\mathcal{X}_0$.

Following the same argument in the proof of Lemma \ref{lemma 6}, we can show that $u_\infty$ cannot be $0$ by the assumption 
$$\frac{1}{V_j} \int_{\mathcal{X}_{t_j}} |u_j| \omega_j^n=1.$$   
On the other hand, $\mathcal{X}_0\setminus \cS$ is connected because $\mathcal{X}_0$ is irreducible and $\mathcal{X}_0\cap \cS$ is a subvariety of $\mathcal{X}$. By the same argument in the proof of Lemma \ref{lemma 6}, we can also show that $u_\infty$ must vanish everywhere in $\mathcal{X}_0\setminus \cS$. This leads to contradiction and we have completed  the proof for the lemma. 
\end{proof}

With Lemma \ref{lemma6v2}, we can complete the proof of Theorem \ref{famdia} by the same argument for Theorem \ref{thm:main1}.


\section{Calabi-Yau fibrations}

In this section, we will apply Theorem \ref{famdia} to collapsing K\"ahler metrics on a fibration of Calabi-Yau manifolds.  Our first result is to extend the results of \cite{L2} for collapsing Ricci-flat K\"ahler metrics on a fibred Calabi-Yau manifold.

\begin{theorem} \label{10cyf} Let $(X, \omega_X)$ be an $n$-dimensional projective Calabi-Yau manifold and $\pi: X \rightarrow Y$ be a holomorphic fibration over a Rieman surface $Y$. Suppose each fibre $X_y=\pi^{-1}(y)$ is normal and has at worst canonical singularities. Let $\omega_Y$ be a fixed K\"ahler metric on $Y$ and let $\omega_\epsilon $ be the unique Ricci-flat K\"ahler metric  in $\in [ \omega_X + \epsilon^{-1} \omega_Y]$ for each $\epsilon \in (0, 1)$. Then there exist $\alpha>0$ and $C>0$ such that for any $\epsilon>0$, any smooth fibre $X_y$ of $\pi$ and any $x\in X_y$, 
$${\rm diam}(X_y, \omega_\epsilon|_{X_y} ) \leq C, $$
$$ \int_{X_y} \big( |G_{y, \epsilon} (x, \cdot)| +|\nabla  G_{y, \epsilon} (x, \cdot)|  \big) \left( \omega_\epsilon|_{X_y}\right)^n + \left(- \inf_{z \in X_y } G_{y, \epsilon}(x, z) \right)  {\rm Vol}_{\omega_\epsilon|_{X_y}} (X_y)\leq  C,  $$
$$\frac{{\mathrm{Vol}}_{\omega_\epsilon|_{X_y}} (B_{\omega_\epsilon|_{X_y}}(x, R))}{\textnormal{Vol}_{\omega_\epsilon|_{X_y}}(X)}\ge C^{-1} R^\alpha, $$
where $G_{y, \epsilon}$ is the Green's function for $(X_y, \omega_\epsilon|_{X_y})$ and $B_{\omega_\epsilon|_{X_y}}(x, R))$ is the geodesic ball in $(X_y, \omega_\epsilon|_{X_y})$. 

\end{theorem}

\begin{proof} 
Let $\eta$ be a non-where vanishing holomorphic volume form on $X$ and let $\eta_y$ be the relative holomorphic volume form defined by $\eta= \eta_y \wedge dy$ locally on $\mathbb{D} \subset Y$.  Then $\omega_\epsilon$ satisfies
$$\omega_\epsilon^n = A_\epsilon \eta\wedge \overline{\eta},$$
where $A_\epsilon = \frac{[\omega_\epsilon]^n}{\int_X \eta\wedge \overline{\eta}}= O\left(\epsilon^{-1}\right)$.

Let $\theta_y= \omega_X|_{X_y}$ be the restriction of $\omega_X$ to $X_y$. It is proved in \cite{L2} (c.f. Proposition 2.1 and Proposition 2.3 in \cite{L2}) that there exist $p>1$ and $C>0$ such that for any $y\in Y$ and $\epsilon\in (0,1)$, we have 
\begin{equation}\label{lpfam}
\int_{X_y} \left| \frac{\eta_y \wedge \overline{\eta_y}}{ (\theta_y)^{n-1}} \right|^p (\theta_y)^{n-1} \leq C, 
\end{equation}
\begin{equation}\label{scwfam}
 \sup_X  {\rm tr}_{\omega_\epsilon}(\epsilon^{-1} \omega_Y) \leq C. 
 \end{equation}

(\ref{scwfam}) implies that on each $X_y$, we have on $X_y$ 
$$ \frac{\omega_\epsilon^{n-1} \wedge \omega_Y}{\omega_X^{n-1}\wedge \omega_Y} =\frac{\omega_\epsilon^{n-1} \wedge \omega_Y}{\omega_\epsilon^n} \frac{\omega_\epsilon^n}{\omega_X ^{n-1}\wedge \omega_Y} \leq C \frac{\eta\wedge \overline{\eta}}{\omega_X^{n-1}\wedge\omega_Y}\leq C \frac{\eta_y \wedge \overline{\eta_y}}{\theta_y^{n-1}}.$$
Then (\ref{lpfam}) immediately gives the uniform upper bound for the $q$-Nash entropy $\mathcal{N}_{X_y, \theta_y, q}(  \omega_\epsilon|_{X_y})$ for any $\epsilon \in (0,1)$ and fixed $q>0$. 

By Theorem 2.5 of \cite{L2} as a refined Schwarz lemma, there  exists $C>0$ such that 

$$\omega_\epsilon \geq C^{-1} \omega_X ,$$
and so $$\left. \frac{\omega_\epsilon^{n-1}}{(\theta_y)^{n-1}} \right|_{X_y}$$ is uniformly bounded away from $0$. 

Therefore we can directly apply Theorem \ref{famdia} to complete the proof of Theorem \ref{10cyf}. 
\end{proof}
 
In the setting of Theorem \ref{10cyf}, it is proved in \cite{L2} that the extrinsic diameter of a smooth fibre $X_y$ is uniformly bounded, i.e., any two points on a smooth fibre $X_y$ can be joined by a path in $X$ with uniformly bounded arc length with respect to $\omega_\epsilon$ for all $\epsilon>0$. The stronger intrinsic diameter bound for $X_y$ is achieved in Theorem \ref{10cyf}.

We will also apply Theorem \ref{famdia} to the long time collapsing solutions of the K\"ahler-Ricci flow. Let $X$ be an $n$-dimensional projective manifold of nef $K_X$. Suppose $K_X$ is semi-ample and the Kodaira dimension of $X$ is one. The pluricanonical map 
$$\pi: X \rightarrow X_{can}$$ is a holomorphic fibration over the canonical model $X_{can}$. The general fibre of $\pi$ is a smooth Calabi-Yau manifold of dimension $n-1$. We consider the K\"ahler-Ricci flow
\begin{equation}\label{unfl11}
\ddt{g}= - {\rm Ric} (g), ~ g(0)=g_0
\end{equation}
for a given K\"ahler metric $g_0$. Then (\ref{unfl11}) admits a smooth solution $g(t)$ for all $t\geq 0$. We would like to investigate the geometric behavior of $g(t)$ near a singular fibre with mild singularities. 

 \begin{theorem} \label{10krfl} Let $X$ be an $n$-dimensional  projective manifold with semi-ample $K_X$ and ${\rm Kod}(X)\leq 1$. Let $g(t)$ be the solution of the K\"ahler-Ricci flow (\ref{unfl11}). Suppose each fibre of $\pi: X \rightarrow X_{can}$ has at worst canonical singularities.  Then there exists $C>0$ such that for any $t>0$, any smooth fibre $X_y$ of $\pi$ and any $x\in X_y$, 
$${\rm diam}(X_y, g(t)|_{X_y} ) \leq C,$$
$$ \int_{X_y} \big( |G_{t, y} (x, \cdot)| + |\nabla  G_{t, y} (x, \cdot)|\big) dV_{g(t)|_{X_y}} +  \left(-  \inf_{z \in X_y } G_{t, y}(x, z) \right)  {\rm Vol}_{g(t)|_{X_y}} (X_y) \leq  C, $$
$$\frac{{\mathrm{Vol}}_{g(t)|_{X_y}} (B_{g(t)|_{X_y}}(x, R))}{\textnormal{Vol}_{g(t)|_{X_y}}(X)}\ge C^{-1} R^\alpha, $$
where $G_{t, y}$ is the Green's function for $(X_y, g(t) |_{X_y})$ and $B_{g(t) |_{X_y}}(x, R))$ is the geodesic ball in $(X_y, \omega_\epsilon|_{X_y})$. 
. 

\end{theorem}

 \begin{proof}  When the Kodaira dimension is $0$, $c_1(X)=0$ and the flow (\ref{unfl11} converges smoothly to a Ricci-flat K\"ahler metric in the initial K\"ahler class. Therefore it suffices to prove the theorem in the case of Kodaira dimension equal to one. 
 
 Let $X_{can}^\circ$ be the regular values of $\pi: X\rightarrow X_{can}$ and $X_{can}\setminus X_{can}^\circ$ is a finite set of points since $X_{can}$ is a smooth Riemann surface.     We follow the same argument in the proof of Theorem \ref{10cyf}. The estimate (\ref{lpfam}) still holds due to the fibration structure of $\pi: X \rightarrow X_{can}$. Since $X_{can}$ is smooth, we can choose a smooth K\"ahler metric $\hat g$ on $X_{can}$ such that $\pi^*\hat g\in [K_X]$. Let $\chi$, $\omega_0$  and $\omega(t)$ be the smooth closed $(1,1)$-forms corresponding to $\pi^*\hat g$,  $g_0$ and $g(t)$.  We let $\omega_t = \omega_0+ t \chi$ and (\ref{unfl11}) can be reduced to a complex Monge-Amp\`ere flow
 $$\ddt{\varphi} = \log \frac{(\omega_t+\ddbar\varphi)^n}{\Omega}, ~\varphi(0)=0,$$
 where $\Omega$ is the volume form on $X$ with $\ddbar\log \Omega= \chi. $

 Since $K_X$ is semi-ample, the Schwarz lemma analogous to (\ref{scwfam}) also holds from the general estimate in \cite{ST0, ST01}, i.e. there exists $C>0$ such that for all $t\geq 0$, 
 $${\rm tr}_{g(t)}( t \chi ) \leq C.$$
 By the same argument in the proof of Theorem \ref{10cyf}, for any fixed  $q>0$, there exists $C=C(g_0, q)>0$ such that for all $t\geq 0$ and $y\in X_{can}$, we have 
 $$\mathcal{N}_{X_y, g_{0,y}, q}\left(g(t)|_{X_y}\right) \leq C, $$
 where $g_{0, y}=g_0|_{X_y}$. 
 Since $X_y$ is normal, the uniform Nash entropy bound (for $q>n$) above implies that there exists $C>0$ such that for all $t\geq 0$ and $y\in X_{can}$, 
 \begin{equation}\label{fbkrfc0}
   \|\varphi(t)|_{X_y} - \sup_{X_y} \varphi(t)|_{X_y}   \| \leq C.
  \end{equation}
 If we let $\hat\varphi(t) = \frac{1}{V_{g_{0,y}}(X_y)} \int_{X_y} \varphi dV_{g_{0,y}}$ be the average of $\varphi$ along the fibre, then $(\varphi(t) - \hat\varphi(t))$ is uniformly bounded due to (\ref{fbkrfc0}).

 Let $$H= \log {\rm tr}_{ g(t)}( t g_0) - 2A  \Big(\varphi - \frac{1}{V_{g_{0,y}}(X_y)} \int_{X_y} \varphi dV_{g_{0, y}} \Big)$$
 for some $A\geq 1$ to be determined.
Since $V_{g_{0,y}}(X_y)$ is independent of $y\in X_{can}$ and $ \|\ddt{\varphi} \|_{L^\infty(X)}$ is uniformly bounded for all $t\geq 0$ by \cite{ST1}, we can apply the similar second order calculations in \cite{ST0} to the evolution of $H$ by
\begin{eqnarray*}
\Box_t H &\leq& C {\rm tr}_{g(t)} (g_0) - 2A {\rm tr}_{g(t)}( g_t) - \frac{2A}{V_{g_{0,y}}(X_y)} {\rm tr}_{g(t)}\Big( \ddbar \int_{X_y} \varphi \omega_0^{n-1} \Big) + CA\\
&\leq& C {\rm tr}_{g(t)} (g_0) - 2A {\rm tr}_{g(t)}( g_t) - \frac{2A}{V_{g_{0,y}}(X_y)} {\rm tr}_{g(t)}\Big(   \int_{X_y}  (\omega -\omega_t) \wedge\omega_0^{n-1} \Big) + CA\\
&\leq& C {\rm tr}_{g(t)} (g_0) - 2A {\rm tr}_{g(t)}( g_t) - \frac{2A}{V_{g_{0,y}}(X_y)} {\rm tr}_{g(t)}\Big(   \int_{X_y}  \omega_0^n \Big) + CA,
\end{eqnarray*}
for some uniform constant $C>0$ independent of $A$.
We define a closed $(1,1)$ form as the push-forward of $\omega_0^n$ to $X_{can}$ by 
$$\hat \chi =\frac{\int_{X_y} \omega_0^n}{{\rm Vol}_{g_{0,y}}(X_y)}.$$
 It is proved in \cite{ST0, ST01} that $$f= \frac{\hat \chi}{\chi} $$ is smooth away from the singular fibres and $f$ is $L^q(X_{can}, \chi)$ integrable for some $q>1$.
There exists a unique solution $\psi \in L^\infty(X_{can})$ to the following linear equation 
$$ \Delta_\chi \psi = f - \bar f, ~ \bar f = \frac{\int_{X_{can}} f \chi}{\int_{X_{can}} \chi}= \frac{\int_{X_{can}} \hat\chi}{\int_{X_{can}} \chi},$$
since $f$ is $L^q$-integrable for some $q>1$. Furthermore, both $f$ and $\psi$ are smooth on $X_{can}^\circ$,  and
$$ \ddbar \psi =\hat\chi - \bar f \chi.$$ 

 Let $\sigma$ be a defining section of points corresponding to $X_{can}\setminus X_{can}^\circ$ and let $h$ be a smooth hermitian metric on the line bundle associated to $[\sigma]$. Then we let 
 $$H_\epsilon =  \log |\sigma|_h^{2\epsilon} {\rm tr}_{g(t)}(t g_0) - 2A \psi - 2A \Big(\varphi - \frac{1}{V_{g_{0,y}}(X_y)} \int_{X_y} \varphi dV_{g_{0,y} |_{X_y}} \Big),$$ 
 for any sufficiently small $\epsilon>0$. 
Then there exists $C>0$ such that for all $t\geq 0$ and $\epsilon \in (0,1)$,  
\begin{eqnarray*}
\Box_t H_\epsilon 
&\leq&   -A {\rm tr}_{g(t)}(g_t) + 2CA
\end{eqnarray*}
for a fixed sufficiently large $A\geq 0$. Applying the maximum principle, $H_\epsilon \leq C$ for a uniform constant $C>0$ independent of $\epsilon\in (0,1)$. Letting $\epsilon \rightarrow 0$, ${\rm tr}_{g(t)} ( g_t)$ is uniformly bounded above, or equivalently, there exists $c>0$ such that for all $t\geq 0$,
$$g\geq c  g_t. $$
Therefore we have derived a uniform positive lower bound of 
$$\frac{( \omega(t) )^{n-1}}{\omega_0^{n-1}}$$ on each fibre $X_y$. 

Combining the above estimates, we can apply Theorem \ref{thm:main1} to complete the proof.
 \end{proof}
 
 Theorem \ref{krfl2} immediately follows from Theorem \ref{10krfl} by parabolic scaling. 
 

\section{Constant scalar curvature K\"ahler metrics with bounded Nash entropy} \label{seccsck}
\setcounter{equation}{0}

In this section, we will obtain geometric estimates for cscK metrics near the canonical class of a smooth minimal model of general type. 

Let $X$ be an $n$-dimensional K\"ahler manifold with big and nef $K_X$, i.e. a minimal model of general type. It is well-known in birational geometry that $K_X$ is semi-ample and the pluricanonical system of $X$ induces a unique surjective birational morphism $\pi: X \rightarrow X_{can}$ from $X$ to its unique canonical model $X_{can}$. In particular, $\pi$ is isomorphic between $X_{can}^\circ$ and $X^\circ$, where $X_{can}^\circ$ is the set of regular values of $\pi$ and $X^\circ=\pi^{-1}(X_{can}^\circ)$.  
We choose $\Omega$ to be a smooth volume form on $X$ on $X$ such that
$\chi= \ddbar \log \Omega \in [K_X]$ is a semi-positive $(1,1)$-form. It is proved in \cite{EGZ} that there exists a unique $\varphi_{KE} \in {\rm PSH}(X, \chi)\cap L^\infty(X) $ solving the complex Monge-Amp\`ere equation
$$(\chi+\ddbar \varphi_{KE})^n = e^{\varphi_{KE}}\Omega, ~\sup_X \varphi_{KE}=0 $$
on $X$.  If we let $\omega_{KE}= \chi + \ddbar \varphi_{KE}$, then $\omega_{KE}$ descends to $X_{can}$ as a K\"ahler-Einstein current since $K_X = \pi^*K_{X_{can}}$. Furthermore, $\omega_{KE}$ is smooth on $X_{can}^\circ$. If we let  $g_{KE}$ be the smooth K\"ahler-Einstein metric on $X_{can}^\circ$ associated to the K\"ahler-Einstein current $\omega_{KE}$, then it is proved in \cite{S1} that the metric completion of $(X_{can}^\circ, g_{KE})$ is a compact metric space homeomorphic to $X_{can}$.

The following result of \cite{JSS} shows that there always exists a unique cscK metric in a K\"ahler class near the canonical class. 

\begin{lemma} \label{lemcsc} For any K\"ahler class $\mathcal{A}$ of $X$, there exists $\delta_0=\delta_0(\mathcal{A})>0$ such that for any $0< \delta < \delta_0$, there exists a unique cscK metric $g_\delta \in [K_X +\delta \mathcal{A}]$. 

\end{lemma} 

Let $\omega_\delta$ be the K\"ahler form associated to $g_\delta$ in Lemma \ref{lemcsc} and $\omega_\mathcal{A} \in [\mathcal{A}]$ be a fixed K\"ahler form. Then we can write
$$\omega_\delta = \chi+ \delta \omega_\mathcal{A} + \ddbar \varphi_\delta$$ 
for a unique $\varphi_\delta\in C^\infty(X)$ with $\sup_X \varphi_\delta=0$. 

The following estimates are proved in \cite{Lx}.

\begin{lemma} \label{lemliu} Let  $\mathcal{A}$ be a K\"ahler class  of $X$ and $\delta_0=\delta_0(\mathcal{A})>0$ as in Lemma \ref{lemcsc}. Then there exists $C>0$ such that   for any $0< \delta < \delta_0$,
$$C^{-1}\leq \inf_X \frac{\omega_\delta^n}{\omega_\mathcal{A}^n} \leq \sup_X \frac{\omega_\delta^n}{\omega_\mathcal{A}^n} \leq C, ~  \|\varphi_\delta \|_{L^\infty(X)} \leq C. $$
Furthermore, $\varphi_\delta$ converges to $\varphi_{KE}$ in $C^\infty (X^\circ)$ as $\delta \rightarrow 0$. 

\end{lemma}

Lemma \ref{lemliu} implies that the cscK metrics $g_\delta$ converge to the K\"ahler-Einstein metrics $g_{KE}$ on $X^\circ$ smoothly. 

\bigskip

\noindent {\it Proof of Theorem \ref{thm:maincsc}.}   The uniform estimates for the diameter, Green's function and volume non-collapsing follows immediately by applying Lemma \ref{lemliu} to Theorem \ref{thm:main1}.  

It remains to prove that $(X, g_\delta)$ converges to $(X_{can}, d_{KE})$ as $\delta \rightarrow 0$, if $X_{can}$ has only isolated singularities.  Suppose not. We can choose a sequence $\delta_j \rightarrow 0$ such that $(X, g_{\delta_j})$ converges to a compact metric space $(Y, d_Y)$ not isomorphic to $(X_{can}, d_{KE})$.   Since $g_{\delta_j}$ converges smoothly to $g_{KE}$ on $X_{can}^\circ$, $(X_{can}^\circ, g_{KE})$ can be embedded in $(Y, d_Y)$ with local isomorphisms. Then there must be a point $p\in Y$ such that 
$$d_Y(p, X_{can}^\circ)>0.$$ 
Otherwise, $(Y, d_Y)$ will be isomorphic to $(X_{can}, d_{KE})$ because $Y$ is compact and $(X_{can}, d_{KE})$ is the compactification of $(X_{can}^\circ, g_{KE})$ by adding finitely many points from $X_{can} \setminus X_{can}^\circ$. We can pick a sequence of points $p_j \in (X, g_{\delta_j})$ such that $p_j$ converges to $p\in Y$ in Gromov-Hausdorff distance. After possibly passing to a sequence, $p_j$ must also converge to a singular point $p_\infty$ of $X_{can}$ with respect to any fixed metric on $X_{can}$ (for example, the Fubini-Study metric from a projective embedding of $X_{can}$). 

We choose an exhaustion of $X_{can}^\circ$ with a sequence of open sets $U_1 \subset\subset U_2 \subset\subset ....\subset\subset U_k \subset\subset ... \subset X_{can}^\circ$ such that 
$$\lim_{k\rightarrow \infty} U_k =X_{can}^\circ, ~ \lim_{k\rightarrow \infty} {\rm Vol}_{g_{KE}}(U_k) = {\rm Vol}_{g_{KE}} (X_{can}^\circ)= [K_X]^n. $$ 
From the above assumption, there exists $\epsilon>0$ such that for any $k>0$, there exists $j>0$ such that 
$$d_{g_{\delta_j}}(p_j, U_k) > \epsilon,  $$
and so $$B_{g_{\delta_j}}(p_j, \epsilon) \subset X\setminus U_k. $$
By the uniform volume non-collapsing, there exists $\delta>0$ such that for any $k>0$, there exists $j>0$ such that 
$$B_{g_{\delta_j}}(p_j, \epsilon) >\delta$$
and $${\rm Vol}_{g_{\delta_j}}(U_k) < [K_X]^n - \delta. $$
This leads to contradiction by choosing sufficiently large $k$.  \qed


\end{document}